\documentclass[reqno,11pt]{amsart}
\usepackage{amsmath,amsfonts,amssymb,amsxtra,latexsym,amscd,enumerate,enumitem,amsthm,verbatim,here}
\usepackage{mathtools}
\usepackage{array}
\usepackage{booktabs}
\usepackage[table,dvipsnames]{xcolor}
\usepackage[normalem]{ulem}
\usepackage{cancel}

\usepackage[margin=1.2in]{geometry}
\numberwithin{equation}{section}

\newcommand{\red}[1]{\textcolor{red}{#1}}

\usepackage{hyperref}
\definecolor{myrefred}{RGB}{180,0,0}
\definecolor{myciteblue}{RGB}{0,70,140} 
\hypersetup{
    colorlinks=true,
    pdfstartview=FitV,
    linkcolor=myrefred,      
    citecolor=myciteblue,    
    urlcolor=black
}
\definecolor{labelkey}{rgb}{0.6,0,0}

\newcommand{\R}{\mathbb{R}}

\newcommand{\C}{\mathbb{C}}

\newcommand{\la}{\langle}
\newcommand{\ra}{\rangle}

\newcommand{\dom}{{\R^3_+} }
\def\tr{\mathrm{Tr}}
\newcommand{\res}{{\mathcal{R}_+}}
\newcommand{\even}{{\mathcal{E}_{\mathrm e}} }
\newcommand{\odd}{{\mathcal{E}_{\mathrm o}}}
\DeclareMathOperator{\sgn}{sgn}
\newcommand{\robin}{{\mathcal{E}_{q}}}

\numberwithin{equation}{section} 

\newtheorem{theorem}{Theorem}

\newtheorem{lemma}{Lemma}[section]
\newtheorem{corollary}[lemma]{Corollary}

\newtheorem{proposition}[lemma]{Proposition}

\theoremstyle{remark}
\newtheorem{remark}[lemma]{Remark}

\begin{document}
\title[Hartree Scattering in Full and Half Spaces]
{A Physical-Space Approach to Modified Scattering for the 3D Hartree Equation: Full-Space and Robin Half-Space}
\author{Beno\^it Pausader}
\address{Brown University}
\email{benoit\_pausader@brown.edu}
\author{Mengyi Xie}
\address{Yale University}
\email{mengyi.xie@yale.edu}

\begin{abstract}
We prove global well-posedness and modified scattering for the three-dimensional Hartree equation for small data in $H^{0,k}$, with $k>\frac12$ on $\R^3$ and $k=1$ on the half-space $\R^3_+$ with a real Robin boundary condition.  For these data, the nonlinear and asymptotic analysis is carried out entirely in physical space, combining slow-growth estimates for Galilean vector fields with a scale-invariant interaction estimate. For the half-space problem, we construct an extension intertwining the Robin and free flows and a boundary-adapted field that overcome the loss of normal Galilean symmetry.

\end{abstract}

\maketitle

\setcounter{tocdepth}{1}

\tableofcontents

\section{Introduction}

\subsection{Asymptotic behavior for the Schr\"odinger-Poisson equation}

We study the three-dimensional Hartree equation with Coulomb interaction, also known as the Schr\"odinger--Poisson system,
\begin{equation}\label{Hartree}
i\partial_t u-\Delta u+\lambda Vu=0,\qquad -\Delta V= |u|^2, \qquad \lambda\in \R, \qquad u(t=0)=u_0,
\end{equation}
on either the full space $X=\R^3$ or the half-space
$X=\R^3_+:=\{x_1>0\}$. In the half-space case, we impose the Robin and Neumann boundary conditions
\begin{equation}\label{BC}
\partial_nu+qu=0,\qquad\partial_nV=0\qquad\hbox{ on }x\in\partial\mathbb{R}_+^3,
\end{equation}
where $q>0$ and $n$ denotes the outward unit normal. 

In this paper, we develop a robust physical-space approach to the asymptotic behavior of Hartree-type equations, which extends to some domains with boundary. We write $H^{0,s}(X):=\{f:\la x\ra^s f\in L^2(X)\}$, with norm $\|f\|_{H^{0,s}}:=\|\la x\ra^s f\|_{L^2}$, and use $H^{s,0}(\R^3)=H^s(\R^3)$. Our first result treats the full-space and Robin half-space problems in the same $H^{0,1}$ framework.

\begin{theorem}\label{thm:dom}
    Fix $\lambda\in\R$ and $0<\eta<1/2$, and, in the half-space case, fix $q>0$. There exists $\varepsilon_0>0$, depending on $(\lambda, \eta)$ in the full space and on $(q,\lambda, \eta)$ in the half-space, such that, if  \begin{equation}\label{eq:small_initial}
        \|\la x\ra u_0\|_{L^2}\leq \varepsilon_0,
    \end{equation}
    then the Hartree equation \eqref{Hartree}, with \eqref{BC} in the half-space, admits a unique global solution in the Duhamel sense,
    \begin{equation*}
        u\in C([0,\infty);L^2(X))\cap L^4_{\mathrm{loc}}([0,\infty);L^3(X)).
    \end{equation*}
    Moreover, there exists an asymptotic state $h_\infty\in L^2(X)$ and a phase $\phi_\infty\in L^{3,\infty}(X)$ with relation 
    \begin{equation}\label{LaplacePhi}
        -\Delta\phi_\infty=\vert h_\infty\vert^2,
    \end{equation}
    where $\Delta$ stands for the Neumann Laplacian in the half-space, such that as $t\to\infty$, the following modified scattering holds:
\begin{equation}\label{AsymptoticProfile}
u(x,t)=t^{-\frac{3}{2}} e^{-i\frac{|x|^2}{4t}}\,h_\infty(x/t)\,
\exp\!\big(i\lambda \log t\,\phi_\infty(x/t)\big)\;+\varepsilon_0\;\mathcal{O}_{L^2(X)}(t^{-\eta}).
\end{equation}

\end{theorem}

\begin{remark} We make the following remarks:
    \begin{itemize}
    \item See Section \ref{subsec:LWP} for the precise notion of low-regularity solutions in the case $X=\mathbb{R}^3_+$.
        \item We also have propagation of regularity. In particular, for $X=\R^3_+$, if we additionally have $u_0\in H^2(\R^3_+)$ with $\tr(\partial_n u_0+q u_0)=0$, then the weak solution is a global strong solution with
\begin{equation*}
    u\in C([0,\infty);H^2(\R^3_+))\cap C^1([0,\infty);L^2(\R^3_+)),
\end{equation*}
which solves \eqref{Hartree}-\eqref{BC} pointwise in time in $L^2_x$ and preserves the Robin boundary condition for all $t\ge0$.
\item The half-space solution exhibits the same logarithmic long-range correction as in the full space, with the limiting phase now determined by the Neumann Coulomb potential of the asymptotic profile.
    \end{itemize}
\end{remark}

In the full space, the spatial localization assumption can be weakened. The following refinement is proved in Appendix \ref{sec:fractional}. We retain the $H^{0,1}$ argument as the main development because its physical-space structure also applies to the Robin problem.

\begin{theorem}\label{thm:full_fractional}
Fix $\beta>1/2$, $\lambda\in\R$ and $0<\eta<1/10$. There exists $\varepsilon_0>0$, depending on $(\beta,\lambda,\eta)$, such that, if
\begin{equation}\label{eq:fractional_initial}
\|\la x\ra^\beta u_0\|_{L^2(\R^3)}\leq\varepsilon_0,
\end{equation}
then \eqref{Hartree} on $\R^3$ admits a unique global solution in the Duhamel sense,
\begin{equation*}
u\in C([0,\infty);L^2(\R^3))\cap L^4_{\mathrm{loc}}([0,\infty);L^3(\R^3)).
\end{equation*}
Moreover, there exist $h_\infty\in L^2(\R^3)$ and $\phi_\infty=(-\Delta)^{-1}|h_\infty|^2\in L^{3,\infty}(\R^3)$ such that, as $t\to\infty$,
\begin{equation}\label{eq:fractional_asymptotics}
u(x,t)=t^{-\frac32}e^{-i\frac{|x|^2}{4t}}h_\infty(x/t)\exp\!\big(i\lambda\log t\,\phi_\infty(x/t)\big)+\varepsilon_0\mathcal O_{L^2(\R^3)}(t^{-\eta}).
\end{equation}
\end{theorem}

For comparison, \cite{HNO} assumes $u_0\in H^{s,0}(\R^3)\cap H^{0,s}(\R^3)$, with $1/2<s<3/2$. For the $L^2$ modified-scattering conclusion above, Theorem \ref{thm:full_fractional} removes the additional Sobolev assumption.  The results in \cite{KP11} and \cite{Van} require $H^{2,0}\cap H^{0,2}$ and $H^{0,\frac32+}$ data, respectively. This result also complements \cite{Nak02a} which obtains modified wave operators in $H^{0,s}$.

This work was motivated by two research directions.

\subsubsection{Scaling-invariant estimates}
In the whole space $X=\mathbb{R}^3$, \eqref{Hartree} is invariant under the scaling
\begin{equation*}
    u(x,t)\,\, \to \,\, \kappa^2 u(\kappa x,\kappa^2t),\qquad V(x,t)\to \kappa^2V(\kappa x,\kappa^2 t)
\end{equation*}
which leaves the $\dot{H}^{-1/2}$-norm of the initial data $u_0$ invariant. The homogeneous weighted norm scales according to
\[
\big\||x|^s\kappa^2u_0(\kappa x)\big\|_{L^2}=\kappa^{1/2-s}\big\||x|^su_0\big\|_{L^2}.
\]
Thus $s=1/2$ is the scaling-critical exponent for spatial localization, and the range in Theorem~\ref{thm:full_fractional} lies strictly above it.

In the focusing case, this threshold is sharp up to the endpoint. Indeed, fix $\lambda<0$ and let $Q$ be the radial, positive, smooth, exponentially decaying ground state satisfying
\[
-\Delta Q+Q=(4\pi)^{-1}\big(|x|^{-1}*Q^2\big)Q;
\]
see \cite{Lieb,MV}. For $\mu>0$, the rescaled ground state
\[
Q_\mu(x):=|\lambda|^{-1/2}\mu^2Q(\mu x)
\]
generates the standing-wave solution $u_\mu(x,t)=e^{-i\mu^2t}Q_\mu(x)$ of \eqref{Hartree}. For every $0\leq s<1/2$ and $0<\mu\leq1$,
\[
\|Q_\mu\|_{H^{0,s}}\leq|\lambda|^{-1/2}\mu^{1/2-s}\|Q\|_{H^{0,s}}\longrightarrow0\qquad\text{as }\mu\downarrow0.
\]
These solutions cannot satisfy \eqref{eq:fractional_asymptotics}; consequently, for $0\leq s<1/2$, arbitrarily small $H^{0,s}$ data need not exhibit the modified scattering asserted in Theorem \ref{thm:full_fractional}. 

The zero Sobolev index in $H^{0,\frac12+}$ is also natural. For the full-space equation, Galilean invariance combined with nonlinear phase separation yields failure of uniform continuity of the solution map on arbitrarily small $H^s(\R^3)$ balls whenever $s<0$. Thus $L^2$ is the threshold for local uniform well-posedness in the Sobolev scale; see Theorem \textcolor{labelkey}{4.1} in \cite{MXZ}.

Understanding the long-time dynamics for solutions of dispersive equations often requires control of scaling-critical or supercritical quantities which remain bounded over the evolution, especially for large-data problems. In the case of linear scattering, this can be provided by conservation laws or appropriate Strichartz norms. For modified scattering, this is often provided by variants of $\Vert\widehat{u}\Vert_{L^\infty}$ which scales like $\dot{H}^{-3/2}$.

One motivation for the present work is to identify useful bounded quantities at the scaling-critical level. Adapting the Fourier-based approach from \cite{KP11}, one may work with $\Vert u\Vert_Z:=\Vert\widehat{u}\Vert_{L^\infty_tL^3_\xi}$ (see \cite{ICERM}). Here, instead, we use an interaction functional in the spirit of \cite{PlaVeg} and rely on the uniform scale-invariant estimate $\sup_tt\Vert V\Vert_{\dot{H}^\frac{3}{2}}$.

\subsubsection{Hartree on domains with boundary}
Another motivation of the present work is to explore and understand the influence of unbounded boundaries on the asymptotic behavior of solutions to dispersive equations, partly informed by a similar study for kinetic equations \cite{HuPaSu}. 

Even for convex domains, the presence of a boundary introduces several difficulties:
\begin{itemize}
    \item Fourier-based methods become less effective, while even defining the linear flow and associated operators becomes more challenging.
    \item Many formulas associated to the asymptotic correction, see \eqref{AsymptoticProfile}, are naturally formulated on the asymptotic cone.
    \item In the absence of the usual symmetries, ``vector-fields'' operators become harder to use and control.
\end{itemize}
These difficulties are related for us since we consider a {\it long range} nonlinearity, and conjugating by the linearized flow still leads to a non-perturbative dynamics.

These are classical problems and some tools have been devised to account for it (see e.g. \cite{Const,PlaVeg}). Here we bypass the second difficulty by working in a cone, but we insist on methods less reliant on the Fourier transform. Even so, the precise choice of boundary condition is guided by two considerations: $(i)$ the ``minimal'' boundary conditions (Dirichlet or Neumann) for $u$ can be seamlessly accounted for by symmetrization; the real Robin condition is the next natural candidate that still leads to a self-adjoint operator, and $(ii)$ the system contains two Laplacians, whose boundary conditions must be compatible with the nonlinearity; the homogeneous Neumann condition for $V$ is chosen so that, for sufficiently regular $u$ and $V$, product rule gives
\begin{equation*}
    \partial_n\left(Vu\right)+q Vu=\left(\partial_nV\right) u+V\left(\partial_n u+qu\right)=0.
\end{equation*}
At the same time, these boundary conditions retain enough structure that, with additional work, the Galilean vector fields still provide useful estimates.

It remains an interesting problem to extend the theory of \eqref{Hartree} to more general domains, at least to smooth unbounded convex domains of the type considered in \cite{HuPaSu}, and to understand how different boundary conditions affect the analysis.


\subsection{Prior work}

Early work for the nonlinear Schr\"odinger equation with convolution nonlinearity
\begin{equation*}
    i\partial_tu-\Delta u =\lambda\big(|x|^{-\gamma}*|u|^2\big)u, \qquad (x,t)\in\R^3\times\R,
\end{equation*}
include \cite{CG,HayaNau2001,HNO,HayaTsu1987}. In particular, modified wave operators for a broad range of $\gamma$ were constructed in \cite{GinOz93,Gv98,Gv99,GV00}. For the three-dimensional Coulomb case, Nakanishi \cite{Nak02a,Nak02b} obtained modified wave operators in sharp weighted $L^2$ spaces. Later refinements and expositions on this problem can be found in \cite{GV151, GV15}.

A key inspiration for this work is the simple and elegant Fourier-based method developed in \cite{KP11}, based on the space-time resonance method of \cite{GMS}. Their approach gives a new proof of long-range scattering, together with an explicit formula for the logarithmic phase, for small initial data in $H^{2,0}\cap H^{0,2}$ in dimension three. More recently, \cite{Van} used the testing by wave packets method and proved modified scattering for $H^{0,\frac{3}{2}+}$ data.

In contrast to the extensive full-space theory, scattering for Hartree on \emph{domains with boundary} remains largely open. We refer to \cite{Ant,BGT,MaCao10,PlaVeg} for related advances in well-posedness and dispersive estimates. To the best of our knowledge, modified scattering for a time-dependent Hartree equation on a domain with boundary has not previously been established.

Related boundary-sensitive long-time problems also arise for kinetic equations with Coulomb fields, where the geometry of the domain affects both the particle characteristics and the Poisson field; see \cite{HuPaSu}. The full-space Vlasov-Hartree system in \cite{HX} provides a related study involving coupled long-range dynamics. 



\subsection{Main ideas and organization}\label{subsec:idea}
The proof of Theorem \ref{thm:dom} has two layers. We first develop a physical-space mechanism for the full-space equation based on a coupled estimate for a Galilean norm and an interaction quantity. We then construct a Robin extension and an adapted Galilean field that preserve this mechanism in the half-space.

\medskip
\noindent \paragraph{\bf The full-space mechanism.}

In the whole space $X=\mathbb{R}^3$, our approach is organized around a combination of the slow growth of a supercritical norm with a scale-invariant estimate for an interaction functional.

The supercritical norm controls the $L^2$ norm of the solution $u(t)$ and of the ``Galilean derivative'' of the solution,
\begin{equation}\label{eq:Galilean}
G(t)=-\nabla_p \Big |_{p=0} \mathcal{G}_p=2t\nabla_x+ix=2te^{-i\sigma(x,t)}\nabla_xe^{i\sigma(x,t)}, \qquad \sigma(x,t)=\frac{|x|^2}{4t},
\end{equation}
which is the infinitesimal generator of the Galilean symmetry of \eqref{Hartree} via the Galilean boost
\[
\mathcal{G}_pu(x,t):=e^{i[t|p|^2-x\cdot p]}u(x-2tp,t),\qquad p\in \R^3.
\]
Control of $Gu$ yields decay estimates for $u$ through the rescaled profile
\begin{equation}\label{RenormalizedProfileg}
    \begin{split}
        \nabla_y g(y,t)=\frac{1}{2}t^\frac{3}{2}\, e^{\,i\frac{|y|^2}{4}t}Gu(ty,t),\qquad g(y,t)\ :=\ t^{\frac{3}{2}}\,e^{\,i\frac{|y|^2}{4}t}\,u(ty,t).
    \end{split}
\end{equation}

Since \eqref{Hartree} is invariant under Galilean boost, $w=Gu$ solves the {\it linearized equation} at $u$:
\begin{equation}\label{LinH}
\begin{split}
\mathcal{L}_uw&:=\left(i\partial_t-\Delta+\lambda V\right)w+\lambda Wu=0,\qquad -\Delta W=2\Re\{\overline{u}w\},
\end{split}
\end{equation}
which admits the favorable energy identity
\begin{equation*}
\begin{split}
\frac{1}{2}\frac{d}{dt}\Vert w\Vert_{L^2}^2&=\Im\langle i\partial_tw,w\rangle =\frac{\lambda}{2\pi}\iint_{\mathbb{R}^3} \frac{\Re\{\overline{u}(x)w(x)\}\Im\{\overline{u}(y)w(y)\}}{\vert x-y\vert}\,dx\,dy.
\end{split}
\end{equation*}
Using the Cauchy-Schwarz inequality, one easily concludes the slow growth of $\Vert w\Vert_{L^2}$, assuming the scale-invariant bound (with $X=\mathbb{R}^3$)
\begin{equation}\label{eq:I}
    \begin{split}
        t^2I(t) \ll1,\qquad I(t) := \iint_{X\times X}\frac{|u(x,t)|^2\,|u(y,t)|^2}{|x-y|^2}\,dx\,dy.
    \end{split}
\end{equation}
Finally, the evolution of $I(t)$ has a complementary structure. More specifically, we observe that $t^2I(t)$ satisfies an ODE whose forcing involves $u$ and $Gu$ (see \eqref{eq:t2I}); the decay obtained from the Galilean estimate makes this forcing integrable in time.

Once the appropriate bootstrap norm is bounded, the asymptotics can be derived by a detailed analysis of the solution evaluated and renormalized along the rays $y=x/t$, see \eqref{eq:h} and the ``effective renormalized phase'' $tV(ty,t)$ (a proxy for $\theta$ in \eqref{eq:h}). Along the way, a new difficulty comes from the low regularity of the limits. For the renormalized profile, we obtain a quantitative convergence rate in a rough norm $H^{-1}$, which can compensate for slow growth in $H^{1/2}$ (see Proposition \ref{prop:L2_profile_convergence}); the phase is shown to converge nicely once we have isolated the equation \eqref{LaplacePhi} (see Lemma \ref{lem:potential_asymptotics}).

The approach above, which avoids the use of the Fourier transform\footnote{However, to compare to the formulas in previous works \cite{HN98,KP11,Van}, one can use the identification $\xi\sim x/t$ coming from the group-velocity heuristics.}, is well suited for domains with boundary, especially for conical domains where $y=x/t$ creates no additional difficulty.

The same viewpoint also motivated the second author's recent work with Huang on Vlasov-Hartree systems \cite{HX}. For Vlasov-Hartree systems, the potential acting on the Hartree component is generated by a kinetic density and must be analyzed simultaneously with the Vlasov evolution; hence, the physical-space argument developed here is an important ingredient. These examples illustrate the flexibility of the method for Hartree equations coupled to additional dynamics or posed in geometries where translation invariance is lost.

\medskip
\noindent\paragraph{\bf The Robin half-space} We next explain how our full-space argument is adapted to the half-space with a Robin boundary condition. For Dirichlet or Neumann boundary conditions, odd or even reflection reduces the problem directly to the full-space equation. The Robin condition is genuinely different: it is not described by a fixed parity under reflection, and the normal Galilean field $G_1$ does not preserve the boundary condition. We therefore introduce the Robin-adapted field
\[
Z_1(t):=U_q(t)(ix_1)U_q(-t),
\]
and estimate the nonlinear evolution of $Z_1u(t)$. A direct operator-level comparison between $Z_1$ and $G_1$ produces a boundary contribution at the \emph{same weighted order} as the leading Galilean term. In particular, this contribution cannot simply be discarded as a lower-order perturbation; its internal structure must be identified before the commutator with the Hartree potential can be estimated.

To accomplish this, we develop a representation of the Robin flow, motivated by Kaikina’s spectral formulation in \cite{Kaikina}, using a boundary-adapted extension that \emph{intertwines the Robin and free Schr\"odinger evolutions}. The resulting decomposition isolates the leading Galilean field and reveals a hidden cancellation in the boundary contribution. This allows us to close the nonlinear estimate for $Z_1u$ and shows that the full-space physical-space mechanism persists even when normal Galilean symmetry is lost.

\medskip

\noindent \paragraph{\bf Weaker localization in the full space.} The proof of Theorem \ref{thm:full_fractional} retains the same rescaled profile and logarithmic phase, but replaces the first-order Galilean estimate by a fractional one. The main additional step is a uniform $L^3$ bound for the profile $g$: a time-dependent frequency cutoff separates a small high-frequency remainder from a low-frequency part whose $L^3$ evolution has integrable errors. This bound closes the fractional Galilean estimate. The renormalized profile then converges in $H^{-2}$, and interpolation with its slowly growing positive Sobolev norm gives strong convergence and identifies the asymptotic phase.

\medskip

We conclude with the organization of the paper. Section \ref{sec:lin} develops the linear Robin theory and constructs the Robin extension and the intertwining identities with the full-space flow. Section \ref{sec:LWP} establishes the dispersive and Strichartz estimates for the Robin propagator. In Section \ref{sec:GWP}, we introduce the adapted Galilean field, close the coupled bootstrap for the Galilean norm and the interaction quantity, and prove global well-posedness. Section \ref{sec:asymptotics} constructs the logarithmic phase and proves modified scattering. These sections establish the Robin part of Theorem \ref{thm:dom}; its full-space specialization is recorded in Remark \ref{rmk:full}. We present this refinement separately in Appendix \ref{sec:fractional}, keeping the physical-space argument and its Robin implementation together in the main body.


\subsection*{Acknowledgments} This project originated in the meetings at the Brown–Yale PDE Seminar. B. Pausader was partially supported by NSF grant DMS-2452275. M. Xie thanks Adam Black, Wenrui Huang, and Wilhelm Schlag for many helpful discussions. M. Xie also gratefully acknowledges the hospitality of the Department of Mathematics at Brown University, where a substantial portion of this work was carried out.

\subsection*{Use of AI} The authors use AI to assist in reviewing previous work, performing checks of calculations, and polishing the manuscript. All mathematical arguments and computations are done and verified by the authors. 


\section{The Robin Laplacian and linear flow}\label{sec:lin}
In this section, we develop the linear framework used throughout the half-space analysis. We first give a precise construction of the Robin Laplacian and its unitary group. We then introduce a boundary-adapted extension that realizes the Robin evolution as the restriction of the free Schr\"odinger evolution on $\R^3$. 


\subsection{Notations}
For a point $x\in \dom$, we sometimes represent it as $x=(x_1,x')$ with $x_1>0$ and $x'=(x_2,x_3)\in\R^2$. We denote by $\res$ the restriction operator from $\R^3$ to $\dom$. The reflection of $x$ across $\partial\dom$ is denoted by $x^*=(-x_1,x')$. 

For a function $f$ on $\dom$, its even and odd extensions across the boundary are defined by 
\[
\even f(x_1,x')=f(|x_1|,x'),\qquad \odd f(x_1,x')=\sgn(x_1) f(|x_1|,x').
\]
In particular, the representation of solutions to the Neumann problem $-\Delta V=\vert u\vert^2$, $\partial_nV=0$ on $\partial\mathbb{R}^3_+$,
\begin{equation}\label{KnKernel}
    \begin{split}
        V(x,t)=\int_{\mathbb{R}^3_+}K_n(x,y)\vert u(y,t)\vert^2dy,\qquad K_n(x,y):=\frac{1}{4\pi}\left(\frac{1}{\vert x-y\vert}+\frac{1}{\vert x^\ast-y\vert}\right),
    \end{split}
\end{equation}
can be written equivalently as
\begin{equation}\label{eq:even-potential}
\even V =\frac{1}{4\pi}\int_{\R^3} \frac{|\even u(y)|^2}{|x-y|} dy
\qquad\text{on }\R^3.
\end{equation}

For later use, if $f,g:\R_+\to \C$, we denote their Volterra convolution also by the normal convolution $*$ and define it by
\[
f*g(a):=\int _0^a f(a-y)g(y)dy, \qquad a>0.
\]
Extending $f$ and $g$ by zero to the negative half-line identifies their Volterra convolution with the restriction of the usual one-dimensional convolution. Consequently, Young's inequality gives
\begin{equation}\label{eq:Volterra_Young}
\|f*g\|_{L^p(\R_+)} \leq \|f\|_{L^1(\R_+)}\|g\|_{L^p(\R_+)}, \qquad 1\leq p\leq\infty.
\end{equation}

The trace operator on $\partial\dom$ is denoted by $\tr$. We use the convention
\[
\langle f,g\rangle_{L^2(\dom)}=\int_{\dom}f\,\overline g\,dx,
\]
which is linear in the first argument. If $T$ is an operator, $D(T)$ denotes its domain. We write $A\lesssim_q B$ when the implicit constant may depend on the fixed Robin parameter $q>0$. We write $A\simeq B$ if both $A\lesssim B$ and $B\lesssim A$, and use $A\simeq_q B$ when the corresponding implicit constants may depend on $q$. For an interval $I\subset\R$, mixed norms are abbreviated as
\[
\|u\|_{L_t^aL_x^b(I\times\dom)}=\Big\|\|u(t,\cdot)\|_{L^b(\dom)}\Big\|_{L^a(I)}\,.
\]
We also set
\[
H^{0,1}(\dom)=\{f\in L^2(\dom):\langle x\rangle f\in L^2(\dom)\}, \qquad
\|f\|_{H^{0,1}}=\|\langle x\rangle f\|_{L^2},
\]
and denote the half space outward unit normal vector as $n=-e_1$. Finally, we write
\[
H_0=-\Delta,\qquad D(H_0)=H^2(\R^3), \qquad U_0(t)=e^{itH_0}.
\]
Thus $U_0(t)$ is the free Schr\"odinger group corresponding to the convention $i\partial_tu-\Delta u=0$.

\medskip

We conclude this subsection by recording several standard analytic facts that will be used repeatedly.

\paragraph{\bf Sobolev interpolation on the half-space.} We use the standard Sobolev spaces on $\dom$, with
\[
H^{-1}(\dom):=(H_0^1(\dom))^*,
\]
where $H_0^1(\dom)$ denotes the closure of $C_c^\infty(\dom)$ in $H^1(\dom)$. If $s_0<s_1$, $0<\vartheta<1$, and $s=(1-\vartheta)s_0+\vartheta s_1$, then 
\begin{equation}\label{eq:half-space-interpolation-identity}
 [H^{s_0}(\dom),H^{s_1}(\dom)]_\vartheta  =H^s(\dom)
\end{equation}
with equivalent norms, where $[\cdot,\cdot]_\vartheta$ denotes complex interpolation; see \cite[Theorem 2.10.1]{Triebel}. In particular, for $-1\leq s\leq 1$,
\begin{equation}\label{eq:half-space-Sobolev-interpolation}
 \|f\|_{H^s(\dom)}\lesssim\|f\|_{H^{-1}(\dom)}^{\frac{1-s}{2}}\|f\|_{H^1(\dom)}^{\frac{1+s}{2}}.
\end{equation}

\medskip

\paragraph{\bf Lorentz spaces.} For $1\leq p<\infty$ and $1\leq r\leq\infty$, we denote by $L^{p,r}(\Omega)$ the usual Lorentz space. If $f^*$ is the decreasing rearrangement of $f$, its Lorentz quasi-norm is given by
\[
 \|f\|_{L^{p,r}(\Omega)}:=\left(\int_0^\infty\bigl[s^{1/p}f^*(s)\bigr]^r\,\frac{ds}{s}\right)^{1/r},\qquad 1\leq r<\infty,
\]
with the usual modification
\[
 \|f\|_{L^{p,\infty}(\Omega)}:=\sup_{s>0}s^{1/p}f^*(s).
\]
In particular, $L^{p,p}=L^p$, and $L^{p,r_1}\hookrightarrow L^{p,r_2}$ whenever $r_1\leq r_2$. We shall use the following special cases of the Lorentz-space H\"older and interpolation inequalities:
\begin{equation}\label{eq:Lorentz-basic-estimates}
 \|fg\|_{L^2} \lesssim \|f\|_{L^{6,2}}\|g\|_{L^{3,\infty}},\qquad \||f|^2\|_{L^{3/2,1}}
 \lesssim \|f\|_{L^{3,2}}^2,
\end{equation}
and
\begin{equation}\label{eq:L32-interpolation}
 \|f\|_{L^{3,2}} \lesssim \|f\|_{L^2}^{1/2}\|f\|_{L^6}^{1/2}.
\end{equation}
Moreover, since $|x|^{-1}\in L^{3,\infty}(\R^3)$, the Lorentz-space convolution inequality gives the endpoint Coulomb estimates
\begin{equation}\label{eq:endpoint-Coulomb-Lorentz}
 \big\||x|^{-1}*f\big\|_{L^{3,\infty}(\R^3)} \lesssim \|f\|_{L^1(\R^3)}, \qquad \big\||x|^{-1}*f\big\|_{L^\infty(\R^3)} \lesssim \|f\|_{L^{3/2,1}(\R^3)}.
\end{equation}
Finally, the refined Sobolev embedding on $\R^3$, followed by extension across the boundary, yields
\begin{equation}\label{eq:Sobolev-Lorentz-half-space}
 \|f\|_{L^{6,2}(\dom)} \lesssim\|\nabla f\|_{L^2(\dom)}, \qquad f\in H^1(\dom).
\end{equation}
We refer to \cite{ONeil} for the general Lorentz-space H\"older and convolution inequalities and to \cite{Peetre} for the general refined Sobolev embedding.

\medskip
\paragraph{\bf Riesz transforms.} For $j=1,2,3$, let
\[
 R_j:=\partial_j|\nabla|^{-1}, \qquad |\nabla|:=(-\Delta)^{1/2},
\]
denote the $j$-th Riesz transform on $\R^3$. It is a Calder\'on--Zygmund operator with kernel $c\,\frac{x_j}{|x|^4}$ and is bounded on $L^p(\R^3)$ for every $1<p<\infty$. Consequently, for $1<p<\infty$ and $1\leq j,k\leq3$,
\begin{equation}\label{eq:Riesz-transform-bounds}
 \|R_jf\|_{L^p(\R^3)}+\|R_jR_kf\|_{L^p(\R^3)}\lesssim_p \|f\|_{L^p(\R^3)}.
\end{equation}
In particular,
\begin{equation}\label{eq:second-Riesz-bound}
 \|\nabla^2(-\Delta)^{-1}f\|_{L^p(\R^3)}\lesssim_p\|f\|_{L^p(\R^3)}.
\end{equation}
We cite \cite{Stein} for these properties. 


\subsection{The Robin Laplacian}\label{subsec:setup} 
For $q>0$, consider the sesquilinear form
\begin{equation}\label{eq:Robin_form}
\mathfrak h_q[f,g] =\int_\dom \nabla f\cdot\overline{\nabla g}\,dx
+q\int_{\partial\dom}\tr f\,\overline{\tr g}\,dS, \qquad D(\mathfrak h_q)=H^1(\dom).
\end{equation}
\begin{proposition}\label{prop:Robin_Laplacian}
The form $\mathfrak h_q$ is densely defined, closed, and nonnegative on $L^2(\dom)$. Its associated operator is the nonnegative self-adjoint operator
\begin{equation}\label{eq:Hq_def}
H_qf=-\Delta f, \qquad D(H_q)=\big\{f\in H^2(\dom):\tr(\partial_nf+qf)=0\big\},
\end{equation}
so that 
\begin{equation}\label{eq:Hq_form_domain}
D(H_q^{\frac{1}{2}})=H^1(\dom), \qquad \mathfrak h_q[f,g] =\langle H_q^{\frac{1}{2}}f,H_q^{\frac{1}{2}}g\rangle_{L^2(\dom)}.
\end{equation}
Consequently, 
\begin{equation}\label{eq:Uq}
U_q(t)=e^{itH_q},\qquad t\in\R,
\end{equation}
is a strongly continuous unitary group on $L^2(\dom)$. In particular,
\begin{equation}\label{eq:Uq_unitary}
U_q(t)^*=U_q(-t), \qquad \|U_q(t)f\|_{L^2(\dom)}=\|f\|_{L^2(\dom)}.
\end{equation}
The group preserves both $H^1(\dom)$ and $D(H_q)$, and
\begin{equation}\label{eq:Uq_commutative}
H_qU_q(t)f=U_q(t)H_qf,
\qquad f\in D(H_q).
\end{equation}
\end{proposition}
\begin{proof}
The trace theorem gives
\[
\|\tr\, f\|_{L^2(\partial\dom)}\lesssim\|f\|_{H^1(\dom)},
\]
so \eqref{eq:Robin_form} is well defined on $H^1(\dom)$. Standard references include \cite{Evans}. It is densely defined since $C_c^\infty(\dom)$ is dense in $L^2(\dom)$, and it is nonnegative following from $q>0$. Furthermore, the form norm 
\[
\|f\|_{\mathfrak h_q}^2 :=\|f\|_{L^2(\dom)}^2+\mathfrak h_q[f,f]
\]
is equivalent to the usual $H^1(\dom)$ norm. Indeed,
\[
\|f\|_{H^1(\dom)}^2 \leq \|f\|_{\mathfrak h_q}^2 \lesssim_q \|f\|_{H^1(\dom)}^2.
\]
Since $H^1(\dom)$ is complete, $\mathfrak h_q$ is closed.   

The representation theorem for densely defined closed nonnegative forms therefore yields a unique nonnegative self-adjoint operator $H_q$ such that
\begin{equation}\label{eq:form-representation}
\mathfrak h_q[f,g]=\langle H_qf,g\rangle_{L^2(\dom)}, \qquad f\in D(H_q),\quad \text{and }\,g\in H^1(\dom),
\end{equation}
and also gives \eqref{eq:Hq_form_domain}. Thus, it remains to identify the operator domain. If $f\in D(H_q)$, testing \eqref{eq:form-representation} against functions in $C_c^\infty(\dom)$ shows that $H_qf=-\Delta f$ in the sense of distributions. The full variational identity \eqref{eq:form-representation} is the weak formulation of the Robin problem; standard elliptic regularity on the half-space therefore gives $f\in H^2(\dom)$; see, for example, \cite{McLean}. Integration by parts on $\dom$ then yields
\[
\mathfrak h_q[f,g] =\int_{\dom}(-\Delta f)\overline g\,dx +\int_{\partial\dom}(\partial_nf+qf)\,\overline{\tr g}\,dS, \qquad g\in H^1(\dom).
\]
Comparing this identity with \eqref{eq:form-representation}, and using the surjectivity of the trace map $H^1(\dom)\to H^{\frac{1}{2}}(\partial\dom)$, gives the boundary condition
\[
\tr(\partial_nf+qf)=0.
\]
Conversely, if $f\in H^2(\dom)$ with this boundary condition, integration by part again gives
\[
\mathfrak h_q[f,g]=\langle-\Delta f,g\rangle_{L^2(\dom)} \qquad\text{for every }g\in H^1(\dom);
\]
hence $f\in D(H_q)$ with $H_qf=-\Delta f$. This proves \eqref{eq:Hq_def}.

Finally, Stone's theorem gives the strongly continuous unitary group \eqref{eq:Uq} and the identities \eqref{eq:Uq_unitary}. The functional calculus for $H_q$ shows that the group preserves $D(H_q^{1/2})=H^1(\dom)$ and $D(H_q)$ and satisfies \eqref{eq:Uq_commutative}.
\end{proof}

We equip $D(H_q)$ with norm
\begin{equation}\label{eq:Hq_norm}
\|f\|_{D(H_q)} := \left( \|f\|_{L^2(\dom)}^2 +\|H_qf\|_{L^2(\dom)}^2
\right)^{\frac{1}{2}}.
\end{equation}
By the domain characterization \eqref{eq:Hq_def} and the standard elliptic estimate for the homogeneous Robin problem, this norm is equivalent to the $H^2(\dom)$ norm on $D(H_q)$:
\begin{equation}\label{eq:Hq_equivalence}
\|f\|_{D(H_q)}\simeq_q\|f\|_{H^2(\dom)}, \qquad f\in D(H_q).
\end{equation}
With Proposition \ref{prop:Robin_Laplacian}, we can now rigorously say that, for $f\in D(H_q)$, the function $u(t)=U_q(t)f$ belongs to
\[
C(\R;D(H_q))\cap C^1(\R;L^2(\dom))
\]
and is the unique strong solution of
\[
i\partial_tu-\Delta u=0,
\qquad \partial_nu=-qu,
\qquad u(0)=f.
\]
For general $f\in L^2(\dom)$, $U_q(t)f$ is the corresponding mild solution.


\subsection{The Robin extension and the intertwining identity}\label{subsec:robin_extension} 
The Robin flow can be related directly to the free evolution on $\R^3$ through an extension in the normal variable. For a function $f$ on $\dom$, define
\begin{equation}\label{eq:Robin-extension}
\robin f(x_1,x') =
\begin{cases}
f(x_1,x'),&x_1\geq0,\\[2mm]
f(-x_1,x')-2q\displaystyle\int_0^{-x_1} e^{q(x_1+y_1)}f(y_1,x')\,dy_1,&x_1<0.
\end{cases}
\end{equation}
By construction, $\res \robin=\operatorname{Id}$. 
\begin{lemma}\label{lem:robin_extension}
    For every $1\leq p\leq\infty$, the map $\mathcal E_q$ satisfies:
\begin{equation}\label{eq:Robin-extension-Lp}
    \|\mathcal E_qf\|_{L^p(\R^3)}\lesssim\|f\|_{L^p(\dom)}.
\end{equation}
It also maps $H^1(\dom)$ continuously into $H^1(\R^3)$ and the identity
\begin{equation}\label{eq:Robin_extension_odd}
    (q-\partial_1)\mathcal E_qf =\odd (q-\partial_1)f, \qquad f\in H^1(\dom),
\end{equation}
holds in $L^2(\R^3)$.
\end{lemma}
\begin{proof}
    For $a>0$, formula \eqref{eq:Robin-extension} becomes
    \begin{equation}\label{eq:Robin-extension-a}
        \robin f(-a,x') =f(a,x')-2q\int_0^a e^{-q(a-y_1)}f(y_1,x')\,dy_1.
    \end{equation}
    This integral term is a Volterra-type convolution in the normal variable with
    \[
    k_q(r)=2qe^{-qr}\mathfrak{1}_{\{r>0\}}, \qquad \|k_q\|_{L^1(\R_+)}=2.
    \]
    Young's inequality in $x_1$, followed by integration in $x'$, proves \eqref{eq:Robin-extension-Lp}. For the $H^1$ bound, we first observe that the tangential differentiation commutes with the extension, so that 
    \[
    \big\|\partial_j \robin f\big\|_{L^2(\R^3)}\lesssim_q \|\partial_j f\|_{L^2(\dom)}, \qquad j=2,3.
    \]
    For the normal derivative, differentiating the extension formula directly yields, for $x_1<0$,
    \begin{equation*}
        \partial_1(\mathcal{E}_qf(x_1,x^\prime))= -\partial_1 f(-x_1,x') +2q f(-x_1,x') -2q^2\int_0^{-x_1}  e^{q(x_1+y_1)}f(y_1,x') \,dy_1.
    \end{equation*}
    The last term is again a one-sided convolution with an $L^1$ kernel, and Young's inequality therefore gives 
    \[
    \big\|\partial_1\robin f\big\|_{L^2(\R^3)} \lesssim_q \|f\|_{H^1(\dom)}.
    \]
    Hence, $\robin :H^1(\dom)\to H^1(\R^3)$ is bounded. For $x_1=-a<0$, direct computation gives
    \[
     (q-\partial_1)\robin f(-a,x') =-(q-\partial_1)f(a,x'),
    \] 
    whereas the two sides agree for $x_1>0$. This proves \eqref{eq:Robin_extension_odd}.
\end{proof}
The identity \eqref{eq:Robin_extension_odd} tells that the Robin boundary condition is converted into oddness after applying $q-\partial_1$. Since the free Schr\"odinger flow preserves oddness, this is the mechanism behind the intertwining representation below.

\begin{proposition}\label{prop:Robin_intertwining}
The extension $\robin$ maps $D(H_q)$ into $H^2(\R^3)$ and satisfies
\begin{equation}\label{eq:flow_intertwining}
    U_0(t)\robin=\robin U_q(t), \qquad\text{on }L^2(\dom),\quad \text{for }t\in \R.
\end{equation}
Consequently, we have
\begin{equation}\label{eq:intertwining}
    H_0\robin f=\robin H_qf, \qquad f\in D(H_q),
\end{equation}
and
\begin{equation}\label{eq:Robin_free_representation}
    U_q(t)=\res U_0(t)\mathcal E_q.
\end{equation}
\end{proposition}
\begin{proof}
    Let $f\in D(H_q)$. Then $f\in H^2(\dom)$ and the boundary condition together with the trace theorem give
    \[
    \tr \big((q-\partial_1)f \big)=0.
    \]
    Therefore the odd extension of $(q-\partial_1)f$ does not introduce singularities and $\odd(q-\partial_1)f$ belongs to $H^1(\R^3)$. \eqref{eq:Robin_extension_odd} then gives
    \[
    (q-\partial_1)\mathcal E_qf =\odd (q-\partial_1)f \in H^1(\R^3).
    \]
    Since $\robin f\in H^1(\R^3)$ from Lemma \ref{lem:robin_extension}, we have
    \[
    \partial_1 \robin f\in H^1(\R^3). 
    \]
    In the tangential directions, $\partial_j$ commutes with $\robin$. Therefore, we have $\robin f\in H^2(\R^3)$. Moreover, Lemma \ref{lem:robin_extension} yields
    \[
    \|\robin f\|_{H^2(\R^3)} \lesssim_q \|f\|_{H^2(\dom)}.
    \]

    Now we define $w(t)=U_0(t)\robin f$. Then the boundedness of $\robin: D(H_q)\to H^2(\R^3)$ gives
    \[
    w\in C\big(\R; H^2(\R^3)\big)\cap C^1\big(\R; L^2(\R^3)\big),
    \]
    and $w$ solves the free Schr\"odinger equation on $\R^3$ with initial data $\robin f$. Since $q-\partial_1$ commutes with $U_0(t)$, \eqref{eq:Robin_extension_odd} implies
    \[
    (q-\partial_1) U_0(t)\robin f=U_0(t)\odd (q-\partial_1)f.
    \]
    Since the free Schr\"odinger propagator preserves oddness in the normal direction, $(q-\partial_1) w(t)$ is odd in $x_1$. The fact that $(q-\partial_1) w(t)\in H^1(\R^3)$ and the trace theorem hence gives
    \[
    \partial_n w(t, 0,x')=-qw(t,0,x').
    \]
    It then follows that $\res w(t)$ is a strong solution to 
    \[
    i\partial_t v-\Delta v=0\,\, \text{in }\dom, \qquad \partial_n v=-qv\,\, \text{on }\partial\dom, \qquad v(t=0)=f.
    \]
    By uniqueness of the strong Robin initial-boundary value problem, we have
    \begin{equation}\label{eq:half_space_equivalence}
        \res U_0(t) \robin f=U_q(t)f, \qquad \text{for }f\in D(H_q).
    \end{equation}
    It remains to recover the identity on the whole space. We use the following characterization of the Robin extension: if $F\in H^1(\R^3)$ satisfies $\res F=g$ and $(q-\partial_1)F$ is odd in $x_1$, then $F=\robin g$. Indeed, one may write
    \[
    W(a,x')=F(-a,x'), \qquad \text{for }a>0;
    \]
    oddness then gives
    \[
    \partial_a W+qW=\partial_a g-qg, \qquad W(0,x')=g(0,x').
    \]
    Solving this first-order equation yields expression
    \[
    F(-a,x')=W(a,x')=g(a,x')-2q\int_0^{a} e^{-q(a-y_1)}g(y_1,x')\,dy_1,
    \]
    which is precisely the definition for $\robin g$. Applying this characterization to $w(t)$, \eqref{eq:half_space_equivalence} gives
    \[
    U_0(t)\robin f=\robin U_q(t)f, \qquad \text{for }f\in D(H_q).
    \]
    Since $\robin f\in H^2=D(H_0)$, differentiating this identity at $t=0$ gives \eqref{eq:intertwining}. Since $D(H_q)$ is dense in $L^2(\dom)$ and $\robin, U_0(t)$ and $U_q(t)$ are bounded on the relevant $L^2$ spaces, this flow identity extends to \eqref{eq:flow_intertwining}. With $\res \robin=\operatorname{Id}$, we get \eqref{eq:Robin_free_representation}. 
\end{proof}
\begin{remark}
The choice of sign $q>0$ is necessary in the exponentially decaying kernel $e^{-q(a-y_1)}$ in \eqref{eq:Robin-extension-a}. Its integrability is what makes $\mathcal E_q$ bounded on all $L^p$ spaces. For $q=0$, \eqref{eq:Robin-extension} reduces to the even extension associated with the Neumann boundary condition.
\end{remark}


\section{Dispersive estimates and local well-posedness}\label{sec:LWP}
In this section, we derive the dispersive and Strichartz estimates for the Robin propagator from the intertwining identity obtained in Section \ref{sec:lin}. We then use these estimates to establish the local theory for the nonlinear problem.

\subsection{Dispersive and Strichartz estimates}\label{subsec:strichartz}
We begin with the dispersive estimate.
\begin{proposition}[Dispersive decay]\label{prop:dispersive_decay}
    For every $t\neq 0$ and $f\in L^1(\dom)$, one has
    \begin{equation}\label{eq:dispersive_decay}
        \big\|U_q(t)f\big\|_{L^\infty(\dom)}\lesssim |t|^{-\frac{3}{2}} \|f\|_{L^1(\dom)}. 
    \end{equation}
\end{proposition}
\begin{proof}
    By \eqref{eq:Robin_free_representation}, the standard dispersive estimate of $U_0$ and \eqref{eq:Robin-extension-Lp}, we get
    \begin{align*}
        \big\|U_q(t)f\big\|_{L^\infty(\dom)}\,\leq\, \big\|U_0(t)\robin f\big\|_{L^\infty(\R^3)}\,\lesssim\, |t|^{-\frac{3}{2}} \big\|\robin f\big\|_{L^1(\R^3)}\,\lesssim\, |t|^{-\frac{3}{2}} \|f\|_{L^1(\dom)}.
    \end{align*}
\end{proof}

The Strichartz estimates follow. We say that a pair $(p,r)$ is Schr\"odinger-admissible in dimension three if
\[
2\leq p, r\leq \infty, \qquad 2/p+3/r=3/2.
\]
We denote the conjugate exponents of $p$ and $r$ by $p'$ and $r'$, respectively.
\begin{proposition}[Strichartz estimate]\label{prop:strichartz}
    Let $I\subset \R$ be an interval, and let $(p,r)$ and $(\tilde p, \tilde r)$ both be Schr\"odinger-admissible pairs. Then the following estimates hold, with constants independent of $I$. 
    \begin{equation}\label{eq:strichartz_homo}
        \big\|U_q(t)f\big\|_{L^p_t L^r_x(I\times \dom)} \lesssim \|f\|_{L^2(\dom)},
    \end{equation}
    \begin{equation}\label{eq:strichartz_inhomo}
        \left\|\int_I U_{q}(t-s) F(s)ds \right\|_{L^p_t L^r_x(I\times \dom)} \lesssim \big\|F\big\|_{L^{\tilde p'}_t L^{\tilde r'}_x(I\times \dom)}\,,
    \end{equation}
    and 
    \begin{equation}\label{eq:strichartz_retarded}
        \left\|\int_{I\cap \{s<t\}} U_q(t-s)F(s)ds \right\|_{L^p_t L^r_x(I\times \dom)} \lesssim \big\|F\big\|_{L^{\tilde p'}_t L^{\tilde r'}_x(I\times \dom)}\,.
    \end{equation}
\end{proposition}
\begin{proof}
    Since $H_q$ is self-adjoint, Proposition \ref{prop:Robin_Laplacian} gives
    \[
    U_q(t)U_q(s)^*=U_q(t-s);
    \]
    hence Proposition \ref{prop:dispersive_decay} gives
    \[
    \big\| U_q(t)U_q(s)^* f\big\|_{L^\infty(\dom)}\lesssim |t-s|^{-\frac{3}{2}}\|f\|_{L^1(\dom)}.
    \]
    Together with the unitarity of $U_q(t)$ on $L^2(\dom)$, these are precisely the hypotheses of the abstract Keel-Tao Strichartz theorem \cite{KeelTao}. It therefore gives the homogeneous estimate \eqref{eq:strichartz_homo} and the retarded estimate \eqref{eq:strichartz_retarded} for the full admissible range, including the endpoint $(p,r)=(\widetilde p,\widetilde r)=(2,6)$. The nonretarded estimate follows from the homogeneous estimate and its dual. 
\end{proof}


\subsection{Local well-posedness}\label{subsec:LWP}
We now apply the Strichartz estimates in Proposition \ref{prop:strichartz} to the nonlinear problem for small time. Solutions of \eqref{Hartree}-\eqref{BC} with $L^2$-data are understood in the Duhamel sense: $u\in C^0_tL^2_x\cap L^4_tL^3_x$ and for $0\le t\le T$,
\begin{equation}\label{eq:duhamel_dom}
    u(t)=U_q(t)u_0+i\lambda\int_0^t U_q(t-s) \left(V(s)u(s)\right) ds.
\end{equation}

\begin{proposition}[Local well-posedness]\label{prop:LWP}
    Fix $q>0$ and $\lambda\in \R$. For every $u_0\in L^2(\dom)$, there exists $T>0$, depending only on $\|u_0\|_{L^2(\dom)}$, such that \eqref{Hartree}-\eqref{BC} admits a unique solution in the Duhamel sense \eqref{eq:duhamel_dom} for time $[0,T]$ and 
    \[
    u\in C([0,T]; L^2(\dom))\cap L^4([0,T]; L^3(\dom)),
    \]
    with conserved mass:
    \begin{equation}\label{eq:mass}
    \mathcal{M}(u(t)):=\int_{\dom} |u(x,t)|^2\,dx=\int_{\dom}\vert u_0\vert^2\,dx.
\end{equation}
    If, in addition, $u_0\in D(H_q)$, then $T$ may be chosen so that the same solution satisfies
    \[
    u\in C([0,T]; D(H_q))\cap C^1([0,T]; L^2(\dom)),
    \]
    and has conserved energy
    \begin{equation}\label{eq:energy}
    \mathcal{E}(u(t)):=\frac{1}{2}\int_\dom |\nabla u(x,t)|^2 dx+\frac{q}{2} \int_{\partial\dom} |u(x,t)|^2 dS+\frac{\lambda}{4}\int _\dom V(x,t)|u(x,t)|^2 dx.
\end{equation}
    In particular, we have $u\in C([0,T]; H^2(\dom))$ and, for every $t\in [0,T]$, 
    \[
    \tr \left(\partial_n u(t)+qu(t)\right)=0. 
    \]
\end{proposition}
\begin{proof}
Set, for simplicity,
\[
X_T = L^\infty([0,T];L^2(\R^3_+)) \cap L^4([0,T];L^3(\R^3_+)),
\]
By \eqref{eq:even-potential}, \eqref{eq:endpoint-Coulomb-Lorentz}, Hardy-Littlewood-Sobolev inequality and interpolation, we bound
\[
\|V(t)\|_{L^\frac{9}{2}(\dom)}\leq \big\|\even V(t)\|_{L^\frac{9}{2}(\R^3)}\lesssim \|\even u(t)\|_{L^\frac{9}{4}(\R^3)}^2\lesssim \|u(t)\|_{L^2(\dom)}^\frac{4}{3}\|u(t)\|_{L^3(\dom)}^\frac{2}{3},
\]
and hence
\[
\|V(t)u(t)\|_{L^\frac{12}{11}_t L^\frac{9}{5}_x([0,T]\times \dom)}\lesssim \|V(t)\|_{L^\frac{3}{2}_t L^\frac{9}{2}_x([0,T]\times \dom)}\|u(t)\|_{L^4_t L^3_x([0,T]\times \dom)}\lesssim T^{\frac{1}{2}}\|u(t)\|_{X_T}^3.
\]
Using Proposition \ref{prop:strichartz} and that $(12,\frac{9}{4})$ is Schr\"odinger-admissible, we have
\begin{align*}
    \left\|U_q(t)u_0+i\lambda\int_0^t U_q(t-s) \left(V(s)u(s)\right) ds\right\|_{X_T}\lesssim_\lambda \|u_0\|_{L^2}+T^\frac{1}{2} \|u\|_{X_T}^3.
\end{align*}
For two functions $u,v\in X_T$, we write $V[u]$ for the potential generated by $u$ and similarly for $V[v]$. Triangle inequality gives
\[
\big\|V[u](t)u(t)-V[v](t)v(t)\big\|_{L^\frac{12}{11}_t L^\frac{9}{5}_x([0,T]\times \dom)}\lesssim T^\frac{1}{2}\|u-v\|_{X_T} \left(\|u\|_{X_T}^2+\|v\|_{X_T}^2\right).
\]
It then follows that the Duhamel map \eqref{eq:duhamel_dom} is a contraction on a sufficiently small ball in $X_T$, provided that $T$ is chosen depending on $\|u_0\|_{L^2}$. This proves existence and uniqueness in $X_T$. The continuity in time with values in $L^2$ and the mass conservation \eqref{eq:mass} follows from the same estimates and approximation by smooth functions. 

\medskip

We now prove the persistence of $D(H_q)$ through the nonlinear flow. For any $x\in \dom$, splitting the Coulomb kernel in \eqref{eq:even-potential} into its near and far parts gives 
\[
|V(x,t)|\lesssim \int_{|x-y|>1} \frac{|\even u(y,t)|^2}{|x-y|}dy+\int_{|x-y|\leq 1}\frac{|\even u(y,t)|^2}{|x-y|}dy \lesssim \|u(t)\|_{L^2(\dom)}^2 + \|u(t)\|_{L^4(\dom)}^2,
\]
since $|x|^{-1}\in L^2_{\mathrm{loc}}(\R^3)$ in three dimensions. \eqref{eq:even-potential} and the Hardy–Littlewood–Sobolev inequality also yield
\[
\big\|\nabla V(t)\big\|_{L^3(\dom)}\lesssim \big\|\nabla (-\Delta)^{-1}\left(|\even u(t)|^2\right)\big\|_{L^3(\R^3)}\lesssim \|u(t)\|_{L^3(\dom)}^2.
\]
On the other hand, \eqref{eq:even-potential} and the boundedness of Riesz transforms imply
\[
\big\| \nabla^2 V(t)\big\|_{L^3(\dom)}\lesssim \big\|\nabla^2 (-\Delta)^{-1} \left(|\even u(t)|^2\right)\big\|_{L^3(\R^3)}\lesssim \|u(t)\|_{L^6(\dom)}^2.
\]
Leibniz's rule and Sobolev embedding therefore give
\begin{align*}
    \big\|V(t)u(t)\big\|_{H^2(\dom)}&\lesssim \|V(t)\|_{L^\infty(\dom)}\|u(t)\|_{H^2(\dom)}+\|\nabla V(t)\|_{L^3(\dom)}\|\nabla u(t)\|_{L^6}\\
    &\quad \quad +\|\nabla^2 V(t)\|_{L^3}\|u(t)\|_{L^6}\,\,\lesssim \,\,\|u(t)\|_{H^1(\dom)}^2\|u(t)\|_{H^2(\dom)}. 
\end{align*}
The same argument applied to differences gives
\[
\big\|V[u](t)u(t)-V[v](t)v(t)\big\|_{H^2(\dom)}\lesssim \|u(t)-v(t)\|_{H^2(\dom)}\left(\|u(t)\|_{H^2(\dom)}^2+\|v(t)\|_{H^2(\dom)}^2\right).
\]
The energy conservation \eqref{eq:energy} follows by direct computation.  We now verify the boundary condition. We first notice that the full-space extension of $V$ is even in the normal direction from \eqref{eq:even-potential}. Its normal derivative is therefore odd, and the regularity above ensures that its trace is well-defined. Hence, $\tr (\partial_n V(t))=0$. For $u\in D(H_q)$, it follows
\[
\partial_n\left(Vu\right)+qVu=\left(\partial_n V\right) u+V\left(\partial_n u+qu\right)=0, \quad \text{on }\partial\dom.
\]
Hence, we have $Vu\in D(H_q)$. 

Using the equivalence \eqref{eq:Hq_equivalence}, the preceding estimates imply 
\[
\|Vu\|_{D(H_q)} \lesssim_q \|u\|_{D(H_q)}^3
\]
and
\[
\|V[u]u-V[v]v\|_{D(H_q)} \lesssim_q \big(\|u\|_{D(H_q)}^2+\|v\|_{D(H_q)}^2\big)
\|u-v\|_{D(H_q)}.
\]
Since $U_q(t)$ is an isometry in the graph norm, the same Duhamel map \eqref{eq:duhamel_dom} is a contraction on a sufficiently small ball in $C([0,T];D(H_q))$, after decreasing $T$ if necessary. This produces a solution in $C([0,T];D(H_q))$. By uniqueness in $L^2$, it agrees with the mild solution constructed above. Finally, we get $u\in C^1([0,T];L^2)$ from the PDE \eqref{Hartree}. This completes the proof. 
\end{proof}

We conclude with the short-time estimates needed in the bootstrap in next section. 
\begin{corollary}\label{cor:short_time_est}
    There exists $\varepsilon_0=\varepsilon_0(q,\lambda)$ such that, if $\|u_0\|_{L^2}\leq \varepsilon_0$, then the corresponding solution exists on $[0,1]$ and satisfies
    \begin{equation}\label{eq:short_time_est}
        \|u\|_{L^\infty_t L^2_x([0,1]\times \dom)}+\|u\|_{L^4_t L^3_x([0,1]\times \dom)}+\|u\|_{L^2_t L^6_x([0,1]\times \dom)}\lesssim \|u_0\|_{L^2}. 
    \end{equation}
\end{corollary}
\begin{proof}
    For sufficiently small $\varepsilon_0$, the contraction argument in the proof of Proposition \ref{prop:LWP} closes with $T=1$. Applying the Strichartz estimates in Proposition \ref{prop:strichartz} gives the stated bounds.
\end{proof}


\section{The Robin-adapted vector field and global well-posedness}\label{sec:GWP}
The normal Galilean vector field $G_1u$ does not preserve the Robin boundary condition, and we replace it by the modification $Z_1u$. The extension constructed in Section \ref{sec:lin} allows us to recover the estimates associated to $Gu$ after the interaction of $Z$ with normal multiplication is understood precisely.


\subsection{The Robin-adapted vector field}\label{subsec:Z}
The expression for the Galilean vector field \eqref{eq:Galilean} implies that the tangential components $G_2$ and $G_3$ preserve the Robin condition. The normal component behaves differently. Indeed, using $\partial_1u=qu$ on the boundary, we obtain 
\[
(-\partial_1+q)G_1u=-2t\partial_1^2u -iu+2tq \partial_1u=2t(q^2-\partial_1^2)u -iu, \quad \text{on }\partial\dom,
\]
which does not vanish in general. Thus $G_1u$ does not generally belong to $D(H_q)$, and a direct boundary energy estimate for $G_1u$ is unavailable.

The usual Galilean vector field can also be expressed in terms of the free Schr\"odinger group:
\begin{equation}\label{eq:G_U0}
    G(t)=U_0(t)(ix)U_0(-t);
\end{equation}
in other words, $G(t)$ can be obtained by transporting the multiplication operator $ix$ under the free Schr\"odinger flow, and it is natural to consider
\begin{equation}\label{eq:Z}
    Z(t):=U_q(t)(ix)U_q(-t).
\end{equation}
By construction, $Z(t)$ is transported by the Robin flow: $Z(t)U_q(s)=U_q(s)Z(t-s)$. Since the boundary and $H_q$ are invariant under translations in the tangential variables, 
\[
Z_j(t)=G_j(t), \qquad \text{for }j=2,3,
\]
only the normal component requires further analysis. We begin by recording the failure of the Robin extension to commute with normal multiplication. 
\begin{lemma}\label{lem:Cq}
    For a function $f$ on $\dom$, we have 
    \begin{equation}\label{eq:commute_x1}
        \robin(x_1f)=|x_1| \robin f+C_q f,
    \end{equation}
    where 
    \begin{equation}\label{eq:Cq}
        C_q f(x_1,x')=
        \begin{cases}
            0, &x_1\geq 0, \\[2mm]
            -2q\displaystyle\int _0^{-x_1} (x_1+y_1)e^{q(x_1+y_1)}f(y_1,x') dy_1, &x_1<0.
        \end{cases}
    \end{equation}
    Moreover, for $1\leq p\leq \infty$, we have
    \begin{equation}\label{eq:bdd_Cq}
        \|C_q f\|_{L^p(\R^3)}\lesssim q^{-1}\|f\|_{L^p(\dom)}.
    \end{equation}
\end{lemma}
\begin{proof}
    The identity \eqref{eq:commute_x1} follows by direct computation. The integral defining $C_q$, after rewriting in $a=-x_1$, is a Volterra convolution with the kernel $2qre^{-qr}\mathfrak{1}_{\{r>0\}}$, whose $L^1$-norm is $\frac{2}{q}$. The asserted $L^p$ boundedness in \eqref{eq:bdd_Cq} then follows from Young’s inequality.
\end{proof}
This defect $C_q$ leads us to the direction of comparison needed later:
\begin{lemma}\label{lem:comparison_GZ}
    For every $t\in\R$ and every $u$ for which the right-hand side is finite, we have
    \begin{equation}\label{eq:comparison_GZ}
        \big\|G(t)u\|_{L^2(\dom)}\lesssim _q \big\|Z(t)u\big\|_{L^2(\dom)}+\|u\|_{L^2(\dom)}. 
    \end{equation}
\end{lemma}
\begin{proof}
    It suffices to consider the normal component. Set $f=U_q(-t)u$ to be the profile under the Robin flow. Since $G_1(t)=2t\partial_1+ix_1$ is a local first-order differential operator with smooth coefficients, so its action on the restriction of a sufficiently regular full-space function depends only on the function in the interior of $\dom$. Thus
    \[
    G_1(t)\res =\res G_1(t).
    \]
    Combining this with \eqref{eq:Robin_free_representation} and \eqref{eq:G_U0}, we have
    \[
    G_1(t)u=G_1(t)\left(\res U_0(t)\robin f \right)=\res G_1(t)U_0(t)\robin f=\res U_0(t)(ix_1 \robin f).
    \]
    Consequently, the unitarity of the free flow on the full space and the unitarity of $U_q(t)$ on the half space, together with Lemma \ref{lem:robin_extension} and Lemma \ref{lem:Cq} yield
    \begin{align*}
        \big\|G_1(t)u\big\|_{L^2(\dom)}&\leq \big\|U_0(t)(ix_1 \robin f)\big\|_{L^2(\R^3)}\leq \big\|\robin(x_1f)\big\|_{L^2(\R^3)}+\|C_qf\|_{L^2(\R^3)}\\
        &\lesssim_q \big\|ix_1 U_q(-t)u\big\|_{L^2(\dom)}+\big\|U_q(-t)u\big\|_{L^2(\dom)} \,=\, \big\|Z_1(t)u\big\|_{L^2(\dom)}+\|u\|_{L^2(\dom)}.
    \end{align*}
\end{proof}
It is important that this argument yields only a norm comparison, not an operator identity between $Z_1$ and $G_1$. Since the restriction is applied after the free evolution, which is nonlocal in the normal variable, one cannot replace $x_1$ by $|x_1|$ inside the free propagator; their difference on the reflected half-space has the same weighted order as the leading Galilean term and is not an $L^2$-bounded perturbation. The Robin extension identity overcomes this difficulty by isolating a bounded reminder $C_q$, thereby connecting $Z$ to $G$, while preserving the boundary structure needed in the commutator estimate to close the bootstrap. 


\subsection{Bootstrap closure and slow growth of vector field}  We use a fixed absolute exponent $\delta\leq \frac{1}{16}$ to close the bootstrap. Throughout the section, we assume the initial data satisfies \eqref{eq:small_initial}. For $0\leq t\leq T$, and given a solution $u$ of \eqref{Hartree}, we consider the following bootstrap assumption
\begin{equation}\label{eq:boot}
    t\,I(t)^\frac{1}{2}+\la t\ra ^{-2\delta}\big\|Zu(t)\big\|_{L^2(\dom)}^2\leq \varepsilon^2,
\end{equation}
where $I(t)$ is given in \eqref{eq:I} and $Z$ in \eqref{eq:Z}, and we improve both terms on the left-hand side under the smallness condition for initial data \eqref{eq:small_initial} and $\varepsilon_0\ll \varepsilon\ll 1$.

\medskip

We first record two estimates used repeatedly below. For $t>0$, \eqref{eq:Galilean} and Sobolev embedding imply
\begin{equation}\label{eq:L6_trick}
    \|u(t)\|_{L^6(\dom)}\lesssim \big\|\nabla\left(e^{i\sigma} u(t)\right)\big\|_{L^2(\dom)}=\frac{1}{2t}\|Gu(t)\|_{L^2(\dom)}.
\end{equation}
\eqref{eq:even-potential} and Sobolev embedding yield
\begin{align*}
    \| \nabla \even V(t)\|_{L^3(\R^3)}\lesssim \big\||\nabla|^{-\frac{1}{2}} \left(|\even u(t)|^2\right)\big\|_{L^2(\R^3)}\simeq \left(\iint_{\R^3\times \R^3}\frac{|\even u(x,t)|^2|\even u(y,t)|^2 }{|x-y|^2} \,dxdy\right)^\frac{1}{2}.
\end{align*}
Splitting this integral according to the signs $x_1$ and $y_1$, and observing that
\begin{equation}\label{eq:reflected_geometry}
    |x^*-y|\geq |x-y|, \qquad x,y\in \dom, 
\end{equation}
shows that this double integral lies between $2I(t)$ and $4I(t)$. Hence
\begin{equation}\label{eq:nabla_Ve}
    \| \nabla \even V(t)\|_{L^3(\R^3)}\lesssim I(t)^\frac{1}{2}. 
\end{equation}

We first estimate the commutator of $Z$ with the Hartree potential. 
\begin{lemma}\label{lem:commutator}
    For $t>0$, we have
    \begin{equation}\label{eq:commutator}
        \big\| \left[Z(t),V(t)\right] u(t)\big\|_{L^2(\dom)}\lesssim _q \left(I(t)^\frac{1}{2}+t^{-1}\|u(t)\|_{L^2(\dom)}^2\right)\|Gu(t)\|_{L^2(\dom)}.
    \end{equation}
\end{lemma}
\begin{proof}
    \emph{Tangential directions.} For $j=2,3$, we have $Z_j(t)=G_j(t)$, and, using \eqref{eq:L6_trick}  and \eqref{eq:nabla_Ve},
    \begin{align*}
        \big\| \left[Z_j(t),V(t)\right] u(t)\big\|_{L^2(\dom)}&=\big\|2t\partial_j V(t)u(t)\big\|_{L^2(\dom)}\lesssim t \big\|\even (\partial_j V(t))\big\|_{L^3(\R^3)}\|u(t)\|_{L^6(\dom)}\\
        &
        \lesssim I(t)^\frac{1}{2} \|Gu(t)\|_{L^2(\dom)}.
    \end{align*}

    \emph{Normal direction.}  We first write
    \begin{equation}\label{eq:J}
        J(t):=U_0(t)(i|x_1|)U_0(-t), \qquad f:=U_q(-t)u.
    \end{equation}
    Using \eqref{eq:Robin_free_representation}, \eqref{eq:flow_intertwining} and Lemma \ref{lem:Cq}, we obtain
    \begin{align*}
        Z_1(t)u& \,=\,\res U_0(t)\left(i|x_1| \robin f\right)+i\res U_0(t)C_q f\,=\,\res J(t)\robin u+i\res U_0(t)C_q U_q(-t)u.
    \end{align*}
    The second term is $L^2$-bounded uniformly in $t$ from the unitarity of $U_0$ and $U_q$ flows and \eqref{eq:bdd_Cq}. We estimate its commutator first:
    \begin{align*}
        \big\|\left[i\res U_0(t)C_q U_q(-t), V(t) \right]u\big\|_{L^2(\dom)}&\lesssim_q \|V(t)\|_{L^\infty}\|u(t)\|_{L^2(\dom)}\\
        &\lesssim_q t^{-1}\|u(t)\|_{L^2(\dom)}^2 \|Gu(t)\|_{L^2(\dom)}.
    \end{align*}
    Here, the last step follows from \eqref{eq:Lorentz-basic-estimates}--\eqref{eq:endpoint-Coulomb-Lorentz} and \eqref{eq:L6_trick}:
    \begin{equation}\label{eq:V_Linfty}
        \begin{split}
            \|V(t)\|_{L^\infty(\dom)}&\leq \|(-\Delta)^{-1}(|\even u(t)|^2)\|_{L^\infty(\R^3)}\lesssim \|\even u(t)\|_{L^{3,2}(\R^3)}^2\\
            &\lesssim \|u(t)\|_{L^2(\dom)} \|u(t)\|_{L^{6}(\dom)}\lesssim t^{-1}\|u\|_{L^2(\dom)}\|Gu\|_{L^2(\dom)}.
        \end{split}
    \end{equation}

    \smallskip

    We now consider the leading part of $Z_1(t)$. We have
    \begin{align*}
        \left[\res J(t)\robin, V(t)\right]u=\res \left[J(t), (\even V)\right]\robin u+\res J(t)\left(\robin (Vu)-(\even V)\robin u\right).
    \end{align*}
    Since $U_0(t)$ is unitary, $U_0(t)XU_0(-t)$ is unitarily equivalent to the multiplication operator $X$. The functional calculus therefore yields the identity
    \[
    J(t)=2it e^{-i\sigma_1}|\partial_1|e^{i\sigma_1}, \qquad \sigma_1=\frac{x_1^2}{4t}.
    \]
    For each fixed $x'$, the one-dimensional homogeneous commutator estimate \cite[Corollary 5.2]{Li} (see also \cite{KP88}) gives
\[
\|[|\partial_1|,W]v\|_{L^2_{x_1}}\lesssim\|\partial_1W\|_{L^3_{x_1}} \|v\|_{L^6_{x_1}}.
\]
H\"older's inequality in $x'$ with exponents $3$ and $6$ then gives
\[
\|[|\partial_1|,W]v\|_{L^2(\R^3)}\lesssim\|\partial_1W\|_{L^3(\R^3)}\|v\|_{L^6(\R^3)}.
\]
Since multiplication operators commute, we apply this to obtain
    \begin{align*}
        \big\| \res \left[J(t), (\even V(t))\right]\robin u\big\|_{L^2(\dom)}&\leq 2t\big\|\left[|\partial_1|, (\even V(t))\right]e^{i\sigma_1}\robin u\big\|_{L^2(\R^3)}\\
        &\lesssim t\|\partial_1 \even V(t)\|_{L^3(\R^3)}\| \robin u(t)\|_{L^6(\R^3)}\lesssim I(t)^\frac{1}{2}\|Gu(t)\|_{L^2(\dom)},
    \end{align*}
    where the last inequality follows from \eqref{eq:L6_trick} and \eqref{eq:nabla_Ve}. 
    
    It remains to estimate the failure of the Robin extension to commute with multiplication by $V$. Set 
    \[
    \mathfrak{E}u:=\robin(Vu)-(\even V)\robin u.
    \]
    The function $\mathfrak{E}u$ vanishes on $x_1>0$. For $a>0$, a direct calculation gives
    \begin{align*}
        \mathfrak{E}u(-a,x')=2q \int _0^{a} e^{-q(a-y_1)}u(y_1,x')\left( V(a,x')-V(y_1,x')\right)dy_1.
    \end{align*}
    This formula illustrates an essential cancellation: the potential appears through a difference across the reflected interval. 

    Since $\mathfrak{E}u$ is supported on $x_1<0$, the kernel of $|\partial_1|$ has no singularity after restricting to $x_1>0$. Hence, we write
    \begin{align*}
        \big| \res J(t)\mathfrak{E}u(x_1,x')\big| \lesssim t\int_{-\infty}^0 \frac{|\mathfrak{E}u(y_1,x')|}{(x_1-y_1)^2}\,dy_1=t\int_0^\infty \frac{|\mathfrak{E}u(-a,x')|}{a}\times \frac{a}{(x_1+a)^2}\,da.
    \end{align*}
    Schur's test, with weights $x_1^{-\frac{1}{2}}$ and $a^{-\frac{1}{2}}$ bounds
    \[
    \big\|\res J(t)\mathfrak{E}u\big\|_{L^2(\dom)} \lesssim t\left\|\frac{\mathfrak{E}u(-a,x')}{a}\right\|_{L^2_{a,x'}(\dom)}\,.
    \]
    By the mean-value formula, we write
    \[
    V(a,x')-V(y_1,x')=(a-y_1)\int_0^1 \partial_1 V\left(y_1+\theta(a-y_1),x'\right) d\theta.
    \]
    Using change of variable $r=a-y_1$, thus $0\le \frac{r}{a}\leq 1$ for $0\le y_1\leq a$ and applying Minkowski’s inequality, we obtain
    \begin{align*}
        &\left\|\frac{\mathfrak{E}u(-a,x')}{a}\right\|_{L^2_{a,x'}(\dom)} \leq 2q \int_0^1 \left\| \int_0^a \frac{r}{a} e^{-qr} u(a-r,x')\partial_1V(a-(1-\theta)r,x') dr     \right\|_{L^2_{a,x'}(\R^3_+)} d\theta\\
        &\qquad \leq 2q \int_0^1 \int_0^\infty e^{-qr} \big\|\mathfrak{1}_{r\leq a} u(a-r,x')\big\|_{L^6_{a,x'}(\dom)}\big\|\mathfrak{1}_{r\leq a}\partial_1 V(a-(1-\theta)r,x')\big\|_{L^3_{a,x'}(\dom)}dr d\theta\\
        &\qquad \lesssim_q \|u(t)\|_{L^6(\dom)}\|\partial_1 V(t)\|_{L^3(\dom)}\lesssim t^{-1}I(t)^{\frac{1} {2}}\|Gu(t)\|_{L^2(\dom)},
    \end{align*}
    from \eqref{eq:L6_trick} and \eqref{eq:nabla_Ve}. Therefore, we have proved
    \begin{equation}
        \big\|\res J(t)\mathfrak{E}u\big\|_{L^2(\dom)}\lesssim I(t)^{\frac{1} {2}}\|Gu(t)\|_{L^2(\dom)}. 
    \end{equation}
    Combining these estimates proves \eqref{eq:commutator}. 
\end{proof}
\begin{remark}\label{rmk:small_time_commutator}
    We applied \eqref{eq:L6_trick} to get the explicit time decay in the second term in \eqref{eq:commutator}. If the $L^6$-norm appearing in the proof are instead left unchanged, the same argument yields
    \begin{equation}\label{eq:small_time_commutator}
        \big\| \left[Z(t),V(t)\right] u(t)\big\|_{L^2(\dom)}\lesssim _q I(t)^\frac{1}{2}\|Gu(t)\|_{L^2(\dom)}+\|u(t)\|_{L^2(\dom)}^2\|u(t)\|_{L^6(\dom)},
    \end{equation}
    which works better for small time propagation. 
\end{remark}

This proposition immediately gives the required energy estimate for $Zu(t)$:
\begin{corollary}\label{cor:improved_Z}
    Under the bootstrap assumption \eqref{eq:boot} for $0\leq t\leq T$ and the smallness assumption on initial data \eqref{eq:small_initial}, we have, 
    \begin{equation}
        \|Zu(t)\|_{L^2(\dom)}\leq K\varepsilon_0\la t\ra ^{K(\varepsilon_0^2+\varepsilon^2)},
    \end{equation}
    with $\varepsilon_0\ll 1$ and $K$ being some implicit constant depending on $q$ and $\lambda$.
\end{corollary}
\begin{proof}
    Since $Z(t)$ is transported by $U_q(t)$, it commutes with the linear evolution. Applying $Z(t)$ to \eqref{Hartree} gives
    \[
    \left(i\partial_t -\Delta +\lambda V\right)Zu=-\lambda \left[Z(t),V\right]u.
    \]
    Since the potential is real-valued and $H_q$ is a self-adjoint operator on $L^2(\dom)$, they do not contribute to the energy estimate. Hence, Lemma \ref{lem:commutator} gives
    \begin{align*}
        \partial_t \|Zu(t)\|_{L^2(\dom)}\lesssim \left(I(t)^\frac{1}{2}+t^{-1}\|u(t)\|_{L^2}^2\right)\|Gu(t)\|_{L^2(\dom)}.
    \end{align*}
    Using Lemma \ref{lem:comparison_GZ} and the conservation of mass \eqref{eq:mass}, we conclude that
    \[
    \partial_t  \|Zu(t)\|_{L^2(\dom)}\lesssim \left(I(t)^\frac{1}{2}+t^{-1}\varepsilon_0^2\right)\left(\varepsilon_0+\|Zu(t)\|_{L^2(\dom)}\right).
    \]
    Gr\"onwall inequality, together with the bootstrap assumption \eqref{eq:boot}, yields
    \begin{align}\label{eq:Zu_large_time}
        \|Zu(t)\|_{L^2(\dom)}\leq \left(\varepsilon_0+\|Zu(1)\|_{L^2(\dom)}\right)t^{K(\varepsilon_0^2+\varepsilon^2)}.  
    \end{align}
    To propagate $\|Zu(t)\|_{L^2}$ for small time interval, we use \eqref{eq:small_time_commutator}. The same energy estimate gives
    \begin{align*}
        \partial_t \|Zu(t)\|_{L^2(\dom)}\lesssim I(t)^\frac{1}{2} \left(\varepsilon_0+\|Zu(t)\|_{L^2}\right)+\varepsilon_0^2 \|u(t)\|_{L^6(\dom)}.
    \end{align*}
    To handle the interaction during small time, we define the full-space interaction quantity as
    \[
    \tilde I(t):=\iint_{\R^3\times \R^3}\frac{|\even u(x,t)|^2 \,|\even u(y,t)|^2}{|x-y|^2}\, dx\,dy.
    \]
    Decomposing the full-space interaction into same-side and opposite-side contributions and using the symmetry of the even extension, together with \eqref{eq:reflected_geometry}, yields 
    \[
    2I(t)\leq \tilde I(t)\leq 4I(t).
    \]
    Consequently, we have
    \begin{align*}
        I(t)\simeq \tilde I(t)\lesssim\|\even u(t)\|_{L^3(\R^3)}^4\lesssim \|u(t)\|_{L^3(\dom)}^4.
    \end{align*}
    Note that $Zu(0)=ixu_0$. Hence, Gr\"onwall inequality, Corollary \ref{cor:short_time_est}, \eqref{eq:small_initial} and H\"older inequality in time imply
    \begin{align*}
        \|Zu(1)\|_{L^2(\dom)} \leq \left(\varepsilon_0+\varepsilon_0^2\|u\|_{L^1_t L^6_x([0,1]\times \dom)}\right)\exp\left(C\|u\|^2_{L^2_t L^3_x([0,1]\times \dom)}\right)\,\lesssim \varepsilon_0. 
    \end{align*}
    Plugging back into \eqref{eq:Zu_large_time} finishes the proof.
\end{proof}

We next derive the evolution of $I(t)$. 
\begin{lemma}\label{lem:improved_I}
    Under the bootstrap assumption \eqref{eq:boot} for $0\leq t\leq T$ and the smallness assumption on initial data \eqref{eq:small_initial}, suppose additionally that
    $K(\varepsilon_0^2+\varepsilon^2)\leq\frac18$,
    where $K$ is the constant in Corollary \ref{cor:improved_Z}, we have, 
    \begin{equation}
        t^2I(t)\leq C\varepsilon_0^4.
    \end{equation}
\end{lemma}
\begin{proof}
We first notice that, with Galilean vector field, one can write
\[
\Im(\overline u\,\nabla u)=\frac{1}{2t}\Big(\Im(\overline u\,G u)-x\,|u|^2\Big).
\]
Differentiating $I(t)$ in $t$ and using the symmetry in $(x,y)$, we get 
\begin{equation*}
    \begin{split}
       \partial_t I(t)
       &=4\iint_{\dom\times \dom}\frac{|u(x)|^2\,\mathrm{div}_y \Im(\overline u\,\nabla_y u)(y)}{|x-y|^2}\,dxdy\\
       &=\frac{2}{t}\iint_{\dom\times \dom}\frac{|u(x)|^2\,\mathrm{div}_y\big(\Im(\overline u\,G u)(y)\big)}{|x-y|^2}\,dxdy\,-\,\frac{2}{t}\iint_{\dom\times \dom}\frac{|u(x)|^2\,\mathrm{div}_y\big(y\,|u(y)|^2\big)}{|x-y|^2}\,dxdy\\
       &=: A(t)+B(t).
    \end{split}
\end{equation*}
Integrating by parts in $y$ yields\footnote{The following computations are presented formally. Here and in the computations for $B(t)$, one of the intermediate $y$-integral should be interpreted as a principal value. Alternatively, one may consider first the interaction kernel $(\kappa^2+\vert x-y\vert^2)^{-1}$ and let $\kappa\to0^+$.}
\begin{align*}
B(t) &=\frac{4}{t}\iint_{\dom\times\dom} |u(x)|^2\,y\cdot\frac{x-y}{|x-y|^4}\,|u(y)|^2\,dydx +\frac{2}{t}\int_{\dom}\int_{\partial\dom}\frac{|u(x)|^2\,|u(y)|^2\,y_1}{|x-y|^2}\,dS_ydx\\
&=-\frac{2}{t}\iint_{\dom\times \dom}\frac{|u(x)|^2\,|u(y)|^2}{|x-y|^2}\,dydx
=-\frac{2}{t}\,I(t),
\end{align*}
where the boundary term vanishes because $y_1=0$, and the last
line follows after a symmetry argument in $x$ and $y$. Similarly, we also have
\begin{align*}
A(t)=&-\frac{4}{t}\iint_{\dom\times \dom} |u(x)|^2\,\Im(\overline uG u)(y)\cdot\frac{x-y}{|x-y|^4}\,dydx\\
&\quad -\frac{2}{t}\int_{\dom}\int_{\partial\dom}\frac{|u(x)|^2}{|x-y|^2}\,\Im\big(\overline uG_1u\big)(y)\,dS_ydx.
\end{align*}
On the boundary, we have $\Im(\bar uG_1u)=2t\Im\left(\bar u\partial_1 u\right)=2t\Im (q|u|^2)$; hence the second term vanishes and we write
\begin{align*}
    \partial_t I(t)=-\frac{2}{t}I(t)+A(t).
\end{align*}
Equivalently, we write
\begin{equation}\label{eq:t2I}
\partial_t\big(t^2 I(t)\big)=t^2A(t).
\end{equation}
To estimate $A(t)$, extend $u$ and $Gu$ by zero to $\R^3$ and denote using $\tilde \cdot$. The kernels $\frac{x_j-y_j}{|x-y|^4}$ are the kernels of the Riesz transforms $R_j$, up to fixed constants. Their $L^p$ boundedness, H\"older inequality and \eqref{eq:L6_trick} give
\begin{align*}
    |A(t)|&\simeq t^{-1}\left|\sum_{j=1}^3 \int_{\R^3}\left|\tilde u(x,t)\right|^2 R_j \left(\Im (\overline{\tilde u}\,\widetilde{G_ju})\right)(x,t)dx\right|\lesssim t^{-1}\|u(t)\|_{L^6(\dom)}^2\|\bar u(t)Gu(t)\|_{L^\frac{3}{2}(\dom)}\\
    &\lesssim t^{-1}\|u(t)\|_{L^6(\dom)}^3\|Gu(t)\|_{L^2(\dom)}\,\lesssim\, t^{-4}\|Gu(t)\|_{L^2(\dom)}^4.
\end{align*}
Integrating \eqref{eq:t2I} and using the improved estimate in Corollary \ref{cor:improved_Z}, we have, for $1\leq t\leq T$, 
\begin{equation}\label{eq:It_est}
\begin{split}
    t^2I(t)&\leq I(1)+C\int_1^t s^{-2}\|Gu(s)\|_{L^2(\dom)}^4ds\\
    &\lesssim I(1)+\varepsilon_0^4 \int_1^\infty s^{-2+4K(\varepsilon_0^2+\varepsilon^2)}\,ds
    \lesssim I(1)+\varepsilon_0^4,
\end{split}
\end{equation}
since $K(\varepsilon_0^2+\varepsilon^2)\leq\frac18$
bounds the last integral by $2$. We are only left with the small time propagation of $I(t)$. Hardy–Littlewood–Sobolev inequality and \eqref{eq:L6_trick} give
\[
I(t)\lesssim \|u(t)\|_{L^2(\dom)}^2\|u(t)\|_{L^6(\dom)}^2\lesssim t^{-2}\|u(t)\|_{L^2(\dom)}^2\|Gu(t)\|_{L^2(\dom)}^2.
\]
Hence, conservation of mass \eqref{eq:mass} and the short-time bound from Corollary \ref{cor:improved_Z} give
\[
t^2I(t)\lesssim \varepsilon_0^4, \qquad\text{for } t\in (0,1].
\]
Plugging the resulting estimate $I(1)\lesssim\varepsilon_0^4$
back into \eqref{eq:It_est} proves the claimed estimate.
\end{proof}

We can now close the bootstrap. Moreover, we state the global well-posedness we obtained. 
\begin{proposition}\label{prop:GWP}
    Every $u_0\in L^2(\dom)$ gives rise to a unique global solution in the Duhamel sense to \eqref{Hartree}--\eqref{BC} with
    \[
    u\in C([0,\infty);L^2(\dom)) \cap L^4_{\mathrm{loc}}([0,\infty);L^3(\dom)).
    \]
    If $u_0\in H^{0,1}$ is sufficiently small according to \eqref{eq:small_initial}, this solution satisfies
    \begin{equation}\label{eq:slow_growth_Gu}
        \|Gu(t)\|_{L^2(\dom)} \lesssim_{q,\lambda}\varepsilon_0\la t\ra^{K\varepsilon_0^2}, \qquad t\geq0,
    \end{equation}
    where $K=K(q,\lambda)>0$ is a constant depending only on $q$ and $\lambda$. For $u_0\in D(H_q)$, the solution also satisfies
    \[
    u\in C([0,\infty);D(H_q)) \cap C^1([0,\infty);L^2(\dom)).
    \]
\end{proposition}
\begin{proof}
    Proposition \ref{prop:LWP} gives a lifespan depending only on the $L^2$-norm of the initial data. Since mass is conserved, the same local existence time can be used at every stage of the evolution and gives a unique global solution. Strichartz estimates also give $u\in L^2_{\mathrm{loc}}([0,\infty);L^6(\dom))$.
    
    Fix some absolute exponent $\delta<\frac1{16}$. Set $\varepsilon=A\varepsilon_0$, where $A$ is a sufficiently large fixed constant depending only on $q$ and $\lambda$, and then choose $\varepsilon_0$ sufficiently small that
    \[
    K(1+A^2)\varepsilon_0^2\leq\delta,
    \]
    with $K$ as in Corollary \ref{cor:improved_Z}. Corollary \ref{cor:improved_Z} and Lemma \ref{lem:improved_I} strictly improve the bootstrap assumption \eqref{eq:boot}. Hence, together with Lemma \ref{lem:comparison_GZ}, we have \eqref{eq:slow_growth_Gu}, after enlarging constant $K$. 
    
    We are only left to prove the persistence of $D(H_q)$-regularity. Suppose now that $u_0\in D(H_q)$. We derive an estimate that rules out finite-time growth of the $H_q$ norm. The boundedness of the Riesz transforms give
    \[
    \big\|\nabla^2V(t)\big\|_{L^2(\dom)}\lesssim \|u(t)\|_{L^4(\dom)}^2 \lesssim \|u(t)\|_{L^2(\dom)}^\frac{1}{2}\|u(t)\|_{L^6(\dom)}^\frac{3}{2}
    \]
    The Hardy--Littlewood--Sobolev inequality also gives
    \[
    \|\nabla V(t)\|_{L^3(\dom)} \lesssim\|u(t)\|_{L^3(\dom)}^2 \lesssim\|u(t)\|_{L^2(\dom)}\|u(t)\|_{L^6(\dom)}.
    \]
    Leibniz's rule and Sobolev embeddings, together with \eqref{eq:V_Linfty} and the preceding estimates, give
    \[
    \|V(t)u(t)\|_{H^2(\dom)} \lesssim\left(\|u(t)\|_{L^2(\dom)}\|u(t)\|_{L^6(\dom)}+\|u(t)\|_{L^2(\dom)}^{1/2}
    \|u(t)\|_{L^6(\dom)}^{3/2}\right)\|u(t)\|_{H^2(\dom)}.
    \]
    As seen in Proposition \red{\ref{prop:LWP}}, $Vu$ satisfies the Robin boundary condition whenever $u$ does.
    Using that $U_q(t)$ is an isometry in the graph norm and \eqref{eq:Hq_equivalence}, we have
     \begin{align*}
        &\|u(t)\|_{D(H_q)}\\
        &\quad\leq \|u_0\|_{D(H_q)}+ C_{q,\lambda} \int_0^t \left(\|u(s)\|_{L^2(\dom)}\|u(s)\|_{L^6(\dom)} + \|u(s)\|_{L^2(\dom)}^\frac{1}{2}\|u(s)\|_{L^6(\dom)}^\frac{3}{2}\right)\|u(s)\|_{D(H_q)}\,ds.
    \end{align*}
    On every finite time interval, the coefficient multiplying $\|u(s)\|_{D(H_q)}$ belongs to $L^1$, by mass conservation and $u\in L^2_{\mathrm{loc}}L^6$. Gr\"onwall's inequality therefore prevents the graph norm from becoming infinite at any finite time. The local strong solution consequently extends globally.
\end{proof}


\section{Modified scattering}\label{sec:asymptotics}
In this section, we prove the modified-scattering statement in Theorem \ref{thm:dom}. Fix the exponent $0<\eta<\frac{1}{2}$ in Theorem \ref{thm:dom}, and choose $0<\delta<\frac{1}{16}$ so that $\frac{1}{2}-2\delta\geq \eta$. This $\delta$ is independent of the initial data. After decreasing $\varepsilon_0$ so that $K\varepsilon_0^2\leq \frac{\delta}{2}$, Proposition \ref{prop:GWP} gives
\[
\|Gu(t)\|_{L^2(\dom)}\lesssim \varepsilon_0\la t\ra ^\frac{\delta}{2}, \qquad \text{and for }t\geq 1,\,\,\,\,\,\, \|u(t)\|_{L^6(\dom)}\lesssim \varepsilon_0t^{-1+\frac{\delta}{2}}.
\]


\subsection{The renormalized profile}
Define the rescaled profile 
\begin{equation}\label{g}
  g(y,t) := t^{\frac{3}{2}}\,e^{\,i\frac{|y|^2}{4}t}\,u(ty,t),\qquad (y,t)\in \R^3_+\times[1,\infty).
\end{equation}
A direct computation using \eqref{Hartree} gives
\begin{equation}\label{eq:dtg}
  i\partial_t g(y,t)=-\lambda V(ty,t)\,g(y,t)\ +\frac{1}{2\sqrt{t}}\nabla_y\!\cdot\!\Big(e^{\,i\frac{|y|^2}{4}t}\,G u(ty,t)\Big).
\end{equation}
This is a perturbed ODE in $g$, which can be integrated by setting
\begin{equation}\label{eq:h}
  h(y,t):= g(y,t)\,e^{-i\lambda\,\theta(y,t)},\qquad \theta(y,t):=\int_{1}^{t}V(sy,s) \, ds.
\end{equation}
to get
\begin{equation}\label{eq:dth}
  i\partial_t h(y,t)\ =\ \frac{1}{2\sqrt{t}}\,e^{-\,i\lambda\,\theta(y,t)}\nabla_y\cdot\!\Big(e^{\,i\frac{|y|^2}{4}t}\,G u(ty,t)\Big).
\end{equation}

\medskip

We first estimate the spatial variation of the phase. 
\begin{lemma}
    For $s\geq 1$, $j=1,2,3$ one has
\begin{equation}\label{eq:dyphi_L3}
  \|\partial_{y_j}(V(sy,s))\|_{L^3_y(\dom)}\ \lesssim \varepsilon_0^2 s^{-1+\frac{\delta}{2}},
\end{equation}
hence,
\begin{equation}\label{eq:dytheta_L3}
  \|\partial_{y_j}\theta(y,s)\|_{L^3_y(\dom)}\ \lesssim \varepsilon_0^2s^{\frac{\delta}{2}}.
\end{equation}
\end{lemma}
\begin{proof}
Differentiating \eqref{eq:even-potential}, we obtain, for $j=1,2,3$,
\[
    \left|\partial_{y_j}(V(ty,t))\right|\lesssim
    t\int_{\R^3} \frac{|\even u(x,t)|^2}{|x-ty|^2}\,dx.
\]
After taking the $L^3_y$-norm, the Hardy-Littlewood-Sobolev inequality gives
\begin{align*}
    \|\partial_{y_j}(V(ty,t))\|_{L^3_y} &\lesssim
    \big\||\nabla|^{-1}|\even u(t)|^2\big\|_{L^3(\R^3)}\lesssim
    \|\even u(t)\|_{L^3(\R^3)}^2 \lesssim
    \|u(t)\|_{L^2(\dom)}
    \|u(t)\|_{L^6(\dom)} \\
    &\lesssim \varepsilon_0^2t^{-1+\delta/2}.
\end{align*}
The estimate for $\theta$ follows by integrating in time and applying Minkowski's inequality.
\end{proof}

We next obtain a convergence of $h$ in a negative Sobolev space. 
\begin{proposition}\label{prop:H_minus1_convergence}
There exists $h_\infty\in H^{-1}(\dom)$ such that
\begin{equation}\label{eq:Hminus1_convergence}
    \|h(t)-h_\infty\|_{H^{-1}(\dom)} \lesssim \varepsilon_0 t^{-1+\delta}, \qquad \text{for }\,\,t\geq1.
\end{equation}
\end{proposition}
\begin{proof}
Let $\varphi\in H_0^1(\dom)$. Pairing \eqref{eq:dth} with $\varphi$ and integrating by parts in $y$, we obtain
\begin{align*}
    i\partial_t\la h(t),\varphi\ra =-\frac{1}{2\sqrt t}
    \int_{\dom}e^{i t\frac{|y|^2}{4}}Gu(ty,t)\cdot
    \nabla_y\left(e^{-i\lambda\theta(y,t)}\overline{\varphi(y)}\right)dy.
\end{align*}
The trace of $\varphi$ vanishes, so there is no boundary term. Since
\[
    \nabla\left(e^{-i\lambda\theta}\bar \varphi\right)=e^{-i\lambda\theta}\left(\nabla\bar \varphi-i\lambda \bar \varphi\nabla\theta\right),
\]
H\"older's inequality and the change of variables by $x=ty$ give
\begin{align*}
    \left|\partial_t\langle h(t),\varphi\rangle\right|\lesssim
    t^{-2}\|Gu(t)\|_{L^2(\dom)}\left(\Vert\nabla\varphi\Vert_{L^2(\dom)}
    +\|\varphi\|_{L^6(\dom)}\|\nabla\theta(t)\|_{L^3(\dom)}\right) \lesssim \varepsilon_0t^{-2+\delta}\|\varphi\|_{H^1(\dom)}.
\end{align*}
Thus
\[
    \|\partial_th(t)\|_{H^{-1}(\dom)} \lesssim \varepsilon_0t^{-2+\delta}.
\]
Since this is integrable on $[1,\infty)$, $h(t)$ converges in $H^{-1}$ to some $h_\infty$, and
\[
    \|h(t)-h_\infty\|_{H^{-1}(\dom)} \lesssim \varepsilon_0\int_t^\infty s^{-2+\delta}\,ds \lesssim \varepsilon_0t^{-1+\delta}.
\]
This proves \eqref{eq:Hminus1_convergence}.
\end{proof}

To upgrade this convergence, we first record the $H^1$-growth of $h$. By mass conservation and the definition of $g$,
\begin{equation}\label{eq:conserved_h}
     \|h(t)\|_{L^2(\dom)} =\|g(t)\|_{L^2(\dom)}=\|u(t)\|_{L^2(\dom)}=\|u_0\|_{L^2(\dom)}.
\end{equation}
Differentiating $g$ gives the exact identity
\[
    \nabla_yg(y,t)=\frac{1}{2}t^{\frac{3}{2}}e^{i t\frac{|y|^2}{4}}(Gu)(ty,t).
\]
Hence
\[
    \Vert g(t)\Vert_{L^6(\mathbb{R}^3_+)}\lesssim \|\nabla g(t)\|_{L^2(\dom)} =\frac{1}{2}\|Gu(t)\|_{L^2(\dom)}\lesssim\varepsilon_0\la t\ra^{\delta/2}.
\]
Since
\[
    \nabla h =e^{-i\lambda\theta} \left(\nabla g-i\lambda g\nabla\theta\right),
\]
we have
\begin{align}\label{eq:h_H1_growth}
    \|h(t)\|_{H^1(\dom)}&\lesssim \|h(t)\|_{L^2(\dom)}+\|\nabla g(t)\|_{L^2(\dom)}+\|g(t)\|_{L^6(\dom)}\|\nabla\theta(t)\|_{L^3(\dom)}\lesssim \varepsilon_0\la t\ra^\delta.
\end{align}
This growth, together with the $H^{-1}$-convergence, gives additional regularity of the limit.

\begin{lemma}\label{lem:limit_regularity}
For every $-1\leq s<1-2\delta$, we have $h_\infty\in H^s(\dom)$ and $\|h_\infty\|_{H^s(\dom)}\lesssim_{s,\delta}\varepsilon_0$.
\end{lemma}
\begin{proof}
Set
\[
    d_n=h(2^{n+1})-h(2^n).
\]
The $H^1$-growth \eqref{eq:h_H1_growth} gives
\[
    \|d_n\|_{H^1}\lesssim\varepsilon_0 2^{n\delta},
\]
while Proposition \ref{prop:H_minus1_convergence} gives
\[
    \|d_n\|_{H^{-1}}\leq\int_{2^n}^{2^{n+1}}\|\partial_t h(\tau)\|_{H^{-1}(\dom)}\,d\tau\lesssim\varepsilon_0 2^{-n(1-\delta)}.
\]
Interpolating between $H^{-1}$ and $H^1$, we obtain
\[
    \|d_n\|_{H^s}\lesssim\|d_n\|_{H^1}^{\frac{s+1}{2}}\|d_n\|_{H^{-1}}^{\frac{1-s}{2}}\lesssim\varepsilon_0 2^{n\frac{s+2\delta-1}{2}}.
\]
The exponent is negative precisely when $s<1-2\delta$; hence, the telescoping series
\[
    h(1)+\sum_{n=0}^\infty d_n
\]
converges in $H^s$, with norm bounded by $C_{s,\delta}\varepsilon_0$. The uniqueness of the $H^{-1}$-limit identifies its limit with $h_\infty$, which proves the claim.
\end{proof}

We can now obtain quantitative $L^2$-convergence.
\begin{proposition}\label{prop:L2_profile_convergence}
Set
\[
    \rho=\frac12-\delta>\frac{3\delta}{2}.
\]
Then
\begin{equation}\label{eq:L2_profile_convergence}
    \|h(t)-h_\infty\|_{L^2(\dom)}\lesssim\varepsilon_0t^{-\rho},\qquad t\geq1.
\end{equation}
In particular, $\rho>\delta$ and $\|h_\infty\|_{L^2}=\|u_0\|_{L^2}$.
\end{proposition}
\begin{proof}
Fix $t\geq1$ and choose an integer $N\geq0$ such that $2^N\leq t<2^{N+1}$. Taking $s=0$ in the proof of Lemma \ref{lem:limit_regularity}, we obtain
\[
    \|h(2^N)-h_\infty\|_{L^2}\leq\sum_{n=N}^\infty\|d_n\|_{L^2}\lesssim\varepsilon_0 2^{-N\rho}.
\]
On the other hand, \eqref{eq:Hminus1_convergence} and \eqref{eq:h_H1_growth} give
\[
    \|h(t)-h(2^N)\|_{H^{-1}} \lesssim\varepsilon_0 2^{-N(1-\delta)},\qquad \|h(t)-h(2^N)\|_{H^1}\lesssim \varepsilon_0 2^{N\delta}.
\]
Interpolating these estimates yields
\begin{align*}
    \|h(t)-h(2^N)\|_{L^2}&\lesssim\|h(t)-h(2^N)\|_{H^{-1}}^{\frac12}\|h(t)-h(2^N)\|_{H^1}^{\frac12}\lesssim\varepsilon_0 2^{-N\frac{1-\delta}{2}}2^{N\frac{\delta}{2}}
    =\varepsilon_0 2^{-N\rho}.
\end{align*}
Consequently,
\[
    \|h(t)-h_\infty\|_{L^2}\leq\|h(t)-h(2^N)\|_{L^2}+\|h(2^N)-h_\infty\|_{L^2}\lesssim\varepsilon_0t^{-\rho}.
\]
This proves \eqref{eq:L2_profile_convergence}. Taking the limit in \eqref{eq:conserved_h} proves the mass identity.
\end{proof}


\subsection{Asymptotics of the phase}\label{subsec:phase_asymptotics}
The homogeneity of the Neumann kernel from \eqref{KnKernel} determines the leading phase explicitly: since
\[
    K_n(ty,tz)=\frac{1}{t}K_n(y,z),
\]
the change of variables $x=tz$ gives
\begin{align}\label{eq:rescaled_V}
    V(ty,t)=\int_{\dom}K_n(ty,x)|u(x,t)|^2\,dx=\frac{1}{t} \int_{\dom}K_n(y,z)|g(z,t)|^2\,dz =
    \frac{1}{t} \int_{\dom}K_n(y,z)|h(z,t)|^2\,dz.
\end{align}
This suggests defining
\begin{equation}\label{eq:phi_infty}
    \phi_\infty(y) =\int_{\dom} K_n(y,z)|h_\infty(z)|^2\,dz.
\end{equation}
With \eqref{eq:reflected_geometry}, we have $ 0\leq K_n(y,z)\lesssim |y-z|^{-1}$; hence, Young's inequality \eqref{eq:endpoint-Coulomb-Lorentz} gives
\[
    \|\phi_\infty\|_{L^{3,\infty}(\dom)} \lesssim\|h_\infty\|_{L^2(\dom)}^2\lesssim\varepsilon_0^2. 
\]
The convergence of $h$ now gives the asymptotic behavior of $V$.
\begin{lemma}\label{lem:potential_asymptotics}
For $t\geq1$, we have
\begin{equation}\label{eq:potential_asymptotics}
    \left\| V(ty,t)-\frac1t\phi_\infty(y) \right\|_{L^{3,\infty}_y(\dom)}
    \lesssim \varepsilon_0^2t^{-1-\rho}.
\end{equation}
\end{lemma}
\begin{proof}
From \eqref{eq:rescaled_V}, we write
\[
    V(ty,t)-\frac1t\phi_\infty(y) =\frac{1}{t} \int_{\dom} K_n(y,z)\left(|h(z,t)|^2-|h_\infty(z)|^2\right)\,dz.
\]
\eqref{eq:reflected_geometry} and Young's inequality imply
\begin{align*}
     \left\| V(ty,t)-\frac1t\phi_\infty \right\|_{L^{3,\infty}(\dom)}&\lesssim\frac{1}{t}
    \big\||h(t)|^2-|h_\infty|^2\big\|_{L^1(\dom)}\\
    &\lesssim t^{-1}\|h(t)-h_\infty\|_{L^2(\dom)}\left(\|h(t)\|_{L^2(\dom)}+\|h_\infty\|_{L^2(\dom)}\right).
\end{align*}
Proposition \ref{prop:L2_profile_convergence} and \eqref{eq:conserved_h} therefore gives the desired estimate.
\end{proof}

Define the time-independent phase correction
\begin{equation}\label{eq:ell}
    \ell(y)=\int_1^\infty\left(V(sy,s)-\frac1s\phi_\infty(y)\right)\,ds.
\end{equation}
Lemma \ref{lem:potential_asymptotics} shows that this integral converges in $L^{3,\infty}(\dom)$. Moreover, we have
\[
    \theta(y,t)-\phi_\infty(y)\log t-\ell(y)=-\int_t^\infty\left(V(sy,s)-\frac1s\phi_\infty(y)\right)\,ds.
\]
Consequently,
\begin{equation}\label{eq:phase_asymptotics}
    \left\|\theta(t)-\phi_\infty\log t-\ell\right\|_{L^{3,\infty}(\dom)}\lesssim \varepsilon_0^2t^{-\rho}.
\end{equation}
Thus the logarithmic term is the only unbounded contribution to the phase; the remaining correction converges to the time-independent function $\ell$. 

We are ready to prove the modified scattering stated in Theorem \ref{thm:dom}. 
\begin{proof}[Proof of Theorem \ref{thm:dom}]
Since $g=he^{i\lambda\theta}$, we write
\begin{align*}
    g(t)-h_\infty e^{i\lambda(\phi_\infty\log t+\ell)}=h(t)\left(e^{i\lambda\theta(t)}-e^{i\lambda(\phi_\infty\log t+\ell)}\right) +\big(h(t)-h_\infty\big)e^{i\lambda (\phi_\infty\log t+\ell)}.
\end{align*}
Using $|e^{ix}-e^{iy}|\leq |x-y|$ for $x,y\in \R$, the Lorentz-space H\"older inequality gives
\begin{align*}
    &\left\|g(t)-h_\infty e^{i\lambda(\phi_\infty\log t+\ell)}\right\|_{L^2(\dom)}\\
    &\qquad\, \lesssim \,\|h(t)\|_{L^{6,2}(\dom)}\left\|\theta(t)-\phi_\infty\log t-\ell\right\|_{L^{3,\infty}(\dom)} +\|h(t)-h_\infty\|_{L^2(\dom)}.
\end{align*}
The Sobolev-Lorentz embedding on the half-space gives
\[
    \|h(t)\|_{L^{6,2}(\dom)} \lesssim \|h(t)\|_{H^1(\dom)}\lesssim\varepsilon_0\la t\ra^\delta.
\]
Therefore,
\[
    \left\|g(t)-h_\infty e^{i\lambda(\phi_\infty\log t+\ell)}\right\|_{L^2(\dom)}
    \lesssim \varepsilon_0t^{-\rho+\delta}.
\]
Set
\[
    \tilde h_\infty(y)=h_\infty(y)e^{i\lambda\ell(y)}.
\]
Then $\tilde h_\infty\in L^2(\dom)$, and
\[
    \phi_\infty(y)=\int_{\dom} K_n(y,z)|\tilde h_\infty|^2\,dz.
\]
Hence, we obtain
\begin{equation}\label{eq:final_profile_asymptotics}
    g(y,t)=\tilde h_\infty(y)e^{i\lambda\log t\,\phi_\infty(y)}+\varepsilon_0\mathcal{O}_{L^2(\dom)}(t^{- \rho+\delta}),
\end{equation}
where $\rho-\delta=\frac12-2\delta\geq\eta$ by our choice of $\delta$. Returning to the original variables gives
\begin{align*}
    u(x,t)=t^{-3/2}e^{-i|x|^2/(4t)}\tilde h_\infty\!\left(\frac{x}{t}\right)\exp\left(
    i\lambda\log t\,\phi_\infty\!\left(\frac{x}{t}\right)\right)+\varepsilon_0\mathcal{O}_{L^2(\dom)}(t^{- \rho+\delta}).
\end{align*}
Renaming $\tilde h_\infty$ as $h_\infty$ gives the modified-scattering statement in Theorem \ref{thm:dom}.
\end{proof}

\begin{remark}[The full-space problem]\label{rmk:full}
    We briefly explain how the argument here simplifies for the full-space part of Theorem \ref{thm:dom}, under the same $H^{0,1}$ assumption.

    \noindent \emph{Linear theory.} The linear theory is immediate. The standard dispersive and Strichartz estimates for $U_0(t)$ give the same $L^2$ local well-posedness argument as in Section \ref{sec:LWP}, without using the Robin extension.

    \noindent \emph{Vector field estimate.} The main simplification occurs in the vector-field estimate. On $\R^3$, the usual Galilean invariance directly leads to \eqref{LinH} for $Gu$. The full-space interaction quantity
    \[
    I_{\mathrm{full}}(t):=\iint_{\R^3\times \R^3}\frac{|u(x,t)|^2\,|u(y,t)|^2}{|x-y|^2}\,dx\,dy,
    \]
    and the estimations for $Z_2$ and $Z_3$ in Lemma \ref{lem:commutator} close the same bootstrap assumption \eqref{eq:boot}. 

    \noindent \emph{Modified scattering.} The modified-scattering argument in Section \ref{sec:asymptotics} then applies without change. The only difference is that all spaces are now taken over $\R^3$. 
\end{remark}


\appendix
\section{Full-space modified scattering with weaker spatial localization}\label{sec:fractional}

In this appendix, we prove Theorem~\ref{thm:full_fractional}. We work only on $\R^3$ and retain the rescaled profile and logarithmic phase used in Section~\ref{sec:asymptotics}. The additional ingredients are a fractional Galilean estimate and a uniform $L^3$ bound for the rescaled profile. All function spaces in this appendix are over $\R^3$, unless otherwise indicated.

Fix $\beta>\frac12$ and $0<\eta<\frac1{10}$ as in Theorem~\ref{thm:full_fractional}. Without loss of generality, we assume that $\beta\leq \frac{3}{4}$. We choose
\begin{equation}\label{eq:frac_parameters}
0<\delta<\min\left\{\frac1{16},\frac{\beta-\frac12}{5},\frac{1-10\eta}{8}\right\}.
\end{equation}
We consider and use the smallness of $\|\la x\ra^\beta u_0\|_{L^2}\leq\varepsilon_0$. Constants below may depend on $\beta$, $\delta$ and $\lambda$, but not on time or $\varepsilon_0$.


\subsection{Fractional estimates and small-time propagation}
We write $|\nabla|^a=(-\Delta)^{\frac{a}{2}}$. Let $\chi\in C_c^\infty(\R^3)$ be real-valued, equal to one on $\{|\xi|\leq1\}$ and supported in $\{|\xi|\leq2\}$, and set $P_{\leq N}=\chi(\frac{D}{N})$ and $P_{>N}=1-P_{\leq N}$. We first record a cutoff commutator estimate.

\begin{lemma}\label{lem:frac_projection_commutator}
Let $0<\kappa<1$, $1<p,q<\infty$ and $1\leq r<\infty$ satisfy $\frac1r=\frac1p+\frac1q$. Then
\begin{equation}\label{eq:frac_projection_commutator}
\|[P_{\leq N},f]v\|_{L^r}\lesssim_{\kappa,p,q}N^{-\kappa}\||\nabla|^\kappa f\|_{L^p}\|v\|_{L^q}.
\end{equation}
\end{lemma}
\begin{proof}
Write $P_{\leq N}v=K_N*v$, where $K_N(y)=N^3K(Ny)$ and $K$ is a Schwartz function. Then
\[
[P_{\leq N},f]v(x)=\int_{\R^3}K_N(y)\big(f(x-y)-f(x)\big)v(x-y)\,dy.
\]
Minkowski's and H\"older's inequalities, together with the fractional translation bound $\|f(\cdot-y)-f\|_{L^p}\lesssim |y|^\kappa\||\nabla|^\kappa f\|_{L^p}$, give
\[
\|[P_{\leq N},f]v\|_{L^r}\lesssim\||\nabla|^\kappa f\|_{L^p}\|v\|_{L^q}\int_{\R^3}|K_N(y)|\,|y|^\kappa\,dy.
\]
The last integral is $N^{-\kappa}\int|K(y)||y|^\kappa\,dy$, which proves the claim.
\end{proof}

We also use the fractional commutator inequality
\begin{equation}\label{eq:frac_KP}
\|[|\nabla|^\beta,f]v\|_{L^2}\lesssim_\beta\||\nabla|^\beta f\|_{L^{{\frac{3}{\beta}}}}\|v\|_{L^{p_*}},\qquad p_*:=\frac6{3-2\beta},
\end{equation}
which follows from the fractional Leibniz rule in \cite[Theorem~1.2]{Li} and H\"older's inequality. The Sobolev and Hardy--Littlewood--Sobolev inequalities give
\begin{equation}\label{eq:frac_Sobolev_HLS}
\|v\|_{L^{p_*}}\lesssim_\beta\||\nabla|^\beta v\|_{L^2},\qquad \big\||\nabla|^\beta(-\Delta)^{-1}|v|^2\big\|_{L^{{\frac{3}{\beta}}}}\lesssim_\beta\|v\|_{L^3}^2.
\end{equation}

Define the fractional Galilean operator
\begin{equation}\label{eq:frac_Gbeta}
G_\beta(t):=U_0(t)|x|^\beta U_0(-t).
\end{equation}
For $t>0$, functional calculus for the free Galilean operators yields
\begin{equation}\label{eq:frac_Gbeta_gauge}
G_\beta(t)=e^{-i\frac{|x|^2}{4t}}(2t|\nabla|)^\beta e^{i\frac{|x|^2}{4t}}.
\end{equation}
In particular, $G_\beta(t)U_0(t-\tau)=U_0(t-\tau)G_\beta(\tau)$ and $G_\beta(0)=|x|^\beta$.

\begin{lemma}\label{lem:frac_short_time}
For sufficiently small $\varepsilon_0$, the full-space solution with $\|\la x\ra^\beta u_0\|_{L^2}\leq\varepsilon_0$ satisfies
\begin{equation}\label{eq:frac_short_time}
\|u\|_{L_t^\infty L_x^2([0,1])}+\|u\|_{L_t^4L_x^3([0,1])}+\|G_\beta u\|_{L_t^\infty L_x^2([0,1])}\lesssim_\beta\varepsilon_0.
\end{equation}
In addition, $G_\beta u\in C([0,1];L^2)$.
\end{lemma}
\begin{proof}
The first two bounds follow from the full-space version of Corollary \ref{cor:short_time_est}.
For $t>0$, set $v(x,t):=e^{i\frac{|x|^2}{4t}}u(x,t)$. By \eqref{eq:frac_Gbeta_gauge} and Sobolev embedding,
\[
\|u(t)\|_{L^{p_*}} =\|v(t)\|_{L^{p_*}}
\lesssim_\beta \||\nabla|^\beta v(t)\|_{L^2}=(2t)^{-\beta}\|G_\beta(t)u(t)\|_{L^2}.
\]
Combining this bound with \eqref{eq:frac_Gbeta_gauge}, the fractional commutator estimate and the Hardy--Littlewood--Sobolev inequality, we obtain
\begin{align}\label{eq:frac_physical_commutator}
\|[G_\beta(t),V(t)]u(t)\|_{L^2}&\lesssim t^\beta \||\nabla|^\beta V(t)\|_{L^{\frac{3}{\beta}}} \|u(t)\|_{L^{p_*}}\lesssim \|u(t)\|_{L^3}^2 \|G_\beta(t)u(t)\|_{L^2}.
\end{align}
Since $G_\beta$ commutes with $i\partial_t-\Delta$, we have
\[
i\partial_t(G_\beta u)-\Delta(G_\beta u)+\lambda V G_\beta u =-\lambda[G_\beta,V]u.
\]
Taking the $L^2$ energy estimate, the reality of $V$ and \eqref{eq:frac_physical_commutator} give
\begin{align*}
\frac12\partial_t\|G_\beta(t)u(t)\|_{L^2}^2
&\leq |\lambda|
\|[G_\beta(t),V(t)]u(t)\|_{L^2}
\|G_\beta(t)u(t)\|_{L^2}\lesssim_{\beta,\lambda} \|u(t)\|_{L^3}^2 \|G_\beta(t)u(t)\|_{L^2}^2.
\end{align*}
Using $G_\beta(0)u_0=|x|^\beta u_0$, Gr\"onwall's inequality yields

\begin{equation}\label{eq:frac_short_gronwall}
\|G_\beta(t)u(t)\|_{L^2}\leq \||x|^\beta u_0\|_{L^2}\exp\left(
C_{\beta,\lambda}\int_0^t\|u(\tau)\|_{L^3}^2\,d\tau \right).
\end{equation}
On $[0,1]$, the integral is bounded by
$\|u\|_{L_t^4L_x^3}^2\lesssim\varepsilon_0^2$,
proving the remaining bound.

A standard approximation argument, using \eqref{eq:frac_physical_commutator} and the local theory in Section \ref{subsec:LWP}, justifies the preceding estimates for the stated initial data and gives $G_\beta u\in C([0,T];L^2)$ for every finite $T$.
\end{proof}


\subsection{Bootstrap estimates}
For $t\geq1$, we work with
\begin{equation}\label{eq:frac_profile}
g(y,t)=t^{\frac{3}{2}}e^{i\frac{|y|^2}{4}t}u(ty,t),\qquad V_g(t):=(-\Delta)^{-1}|g(t)|^2.
\end{equation}
From the definition of $g$, \eqref{eq:frac_Gbeta_gauge} and scaling of the fractional derivative imply
\[
G_\beta(t){u(x,t)}=2^\beta t^{-\frac{3}{2}}e^{-i\frac{|x|^2}{4t}}\left(|\nabla|^\beta g(t)\right)\left(\frac{x}{t}\right).
\]
Mass conservation and a change of variables therefore give
\begin{equation}\label{eq:frac_profile_norms}
\|g(t)\|_{L^2}=\|u_0\|_{L^2},\qquad \|G_\beta(t)u(t)\|_{L^2}=2^\beta\||\nabla|^\beta g(t)\|_{L^2}.
\end{equation}
In particular, Lemma \ref{lem:frac_short_time} and Sobolev embedding yield $\|g(1)\|_{H^\beta}+\|g(1)\|_{L^3}\lesssim\varepsilon_0$. The equation for $g$ is
\begin{equation}\label{eq:frac_profile_equation}
i\partial_tg-t^{-2}\Delta g+\lambda t^{-1}V_g g=0.
\end{equation}
We use the bootstrap assumption
\begin{equation}\label{eq:frac_bootstrap}
\|g(t)\|_{L^3}+t^{-\delta}\||\nabla|^\beta g(t)\|_{L^2}\leq\varepsilon,\qquad 1\leq t\leq T,
\end{equation}
where $\varepsilon=A\varepsilon_0$ and $A$ represents an implicit large constant depending on $\beta,\lambda$ and $\delta$. This $L^3$ bound controls the interaction quantity from the main argument, since
\[
t^2I(t)=\iint_{\R^3\times\R^3}\frac{|g(y,t)|^2|g(z,t)|^2}{|y-z|^2}\,dy\,dz\lesssim\|g(t)\|_{L^3}^4.
\]

\begin{lemma}\label{lem:frac_improved_derivative}
Under \eqref{eq:frac_bootstrap},
\begin{equation}\label{eq:frac_improved_derivative}
\||\nabla|^\beta g(t)\|_{L^2}\leq\||\nabla|^\beta g(1)\|_{L^2}t^{C_{\beta,\lambda}\varepsilon^2},\qquad 1\leq t\leq T.
\end{equation}
\end{lemma}
\begin{proof}
Apply $|\nabla|^\beta$ to \eqref{eq:frac_profile_equation}. The Laplacian and the real multiplication term do not contribute to the $L^2$ energy estimate, so \eqref{eq:frac_KP}--\eqref{eq:frac_Sobolev_HLS} give
\begin{align*}
\partial_t\||\nabla|^\beta g(t)\|_{L^2}&\lesssim_\lambda t^{-1}\|[|\nabla|^\beta,V_g]g\|_{L^2}\\
&\lesssim_{\beta,\lambda}t^{-1}\||\nabla|^\beta V_g\|_{L^{{\frac{3}{\beta}}}}\|g\|_{L^{p_*}}\lesssim_{\beta,\lambda}t^{-1}\|g\|_{L^3}^2\||\nabla|^\beta g\|_{L^2}.
\end{align*}
Gr\"onwall's inequality proves \eqref{eq:frac_improved_derivative}. These energy estimates can be justified by the smooth approximation in Lemma \ref{lem:frac_short_time}.
\end{proof}

Choose $\varepsilon_0$ sufficiently small, this lemma then gives 
\begin{equation}\label{eq:frac_derivative_growth}
\||\nabla|^\beta g(t)\|_{L^2} \lesssim\varepsilon_0t^\delta.
\end{equation}

\begin{lemma}\label{lem:frac_improved_L3}
Under \eqref{eq:frac_bootstrap}, 
\[
\|g(t)\|_{L^3}\lesssim_{\beta,\delta,\lambda}\varepsilon_0, \quad 1\leq t\leq T.
\]
\end{lemma}
\begin{proof}
Split $g=P_{\leq\sqrt t}g+P_{>\sqrt t}g$. Sobolev embedding and the support of the high-frequency multiplier give
\begin{equation}\label{eq:frac_high_frequency}
\|P_{>\sqrt t}g(t)\|_{L^3}\lesssim\||\nabla|^{\frac{1}{2}}P_{>\sqrt t}g(t)\|_{L^2}\lesssim t^{-{\frac{2\beta-1}{4}}}\||\nabla|^\beta g(t)\|_{L^2}\lesssim\varepsilon_0t^{-{\frac{2\beta-1}{4}}+\delta}.
\end{equation}
For $g_{\leq\sqrt t}:=P_{\leq\sqrt t}g$, equation \eqref{eq:frac_profile_equation} gives
\begin{equation}\label{eq:frac_low_equation}
i\partial_tg_{\leq\sqrt t}-t^{-2}\Delta g_{\leq\sqrt t}+\lambda t^{-1}V_g g_{\leq\sqrt t}=i(\partial_tP_{\leq\sqrt t})g-\lambda t^{-1}[P_{\leq\sqrt t},V_g]g.
\end{equation}
Since $V_g$ is real, differentiating the $L^3$ norm yields
\begin{align}\label{eq:frac_low_L3_ode}
\partial_t\|g_{\leq\sqrt t}(t)\|_{L^3}&\leq t^{-2}\|\Delta g_{\leq\sqrt t}\|_{L^3}+\|(\partial_tP_{\leq\sqrt t})g\|_{L^3}+|\lambda|t^{-1}\|[P_{\leq\sqrt t},V_g]g\|_{L^3}.
\end{align}
Bernstein's inequality gives
\begin{equation}\label{eq:frac_low_linear_errors}
t^{-2}\|\Delta g_{\leq\sqrt t}\|_{L^3}+\|(\partial_tP_{\leq\sqrt t})g\|_{L^3}\lesssim t^{-1-{\frac{2\beta-1}{4}}}\||\nabla|^\beta g(t)\|_{L^2}\lesssim\varepsilon_0t^{-1-{\frac{2\beta-1}{4}}+\delta}.
\end{equation}
For the second term, we used that $\partial_t\chi(\frac{\xi}{\sqrt t})$ is $t^{-1}$ times a smooth multiplier supported where $|\xi|\simeq\sqrt t$. For the commutator, choose $\kappa=3\beta-\frac{3}{2}={\frac{6(2\beta-1)}{4}}$ and $p=\frac{6}{2\beta-1}$. Our choice $\beta\leq\frac34$ ensures $0<\kappa<1$, and $\frac{1}{3}=\frac{1}{p}+\frac{1}{p_*}$. Hardy-Littlewood-Sobolev gives
\[
\||\nabla|^\kappa V_g\|_{L^p} \lesssim\||g|^2\|_{L^{\frac{p_*}{2}}}=\|g\|_{L^{p_*}}^2,
\]
since $\frac{1}{p}=\frac{2}{p_*}-\frac{2-\kappa}{3}$. Lemma \ref{lem:frac_projection_commutator} therefore yields
\begin{equation}\label{eq:frac_low_nonlinear_error}
    \begin{split}
        t^{-1}\|[P_{\leq\sqrt t},V_g]g\|_{L^3}&\lesssim t^{-1-\frac{\kappa}{2}}\||\nabla|^\kappa V_g\|_{L^p}\|g\|_{L^{p_*}}\\
        &\lesssim t^{-1-{\frac{3(2\beta-1)}{4}}}\||\nabla|^\beta g\|_{L^2}^3\lesssim\varepsilon_0^3t^{-1-{\frac{3(2\beta-1)}{4}+3\delta}}.
    \end{split}
\end{equation}
By \eqref{eq:frac_parameters}, $\delta<{\frac{2\beta-1}{4}}$, so both errors are integrable on $[1,\infty)$. Integrating \eqref{eq:frac_low_L3_ode}, and using the $L^3$ boundedness of $P_{\leq1}$ and \eqref{eq:frac_high_frequency}, gives
\[
\|g(t)\|_{L^3}\lesssim\|g(1)\|_{L^3}+C_{\beta,\delta,\lambda}(\varepsilon_0+\varepsilon_0^3)\lesssim\varepsilon_0,
\]
which proves the stated estimate. 
\end{proof}

Hence, with $A$ picked sufficiently large compared with the constants in the preceding estimates, Lemmas \ref{lem:frac_improved_derivative} and \ref{lem:frac_improved_L3} strictly improve \eqref{eq:frac_bootstrap}. The local weighted theory gives continuity of $g$ in $H^\beta$ on compact subintervals of $[1,\infty)$, and hence in $L^3$. The continuity argument therefore proves
\begin{equation}\label{eq:frac_global_bounds}
\sup_{t\geq1}\|g(t)\|_{L^3}\lesssim\varepsilon_0,\qquad \|g(t)\|_{H^\beta}\lesssim\varepsilon_0t^\delta,\qquad t\geq1.
\end{equation}

\subsection{Convergence of the renormalized profile} As in Section \ref{sec:asymptotics}, define
\begin{equation}\label{eq:frac_renormalization}
\theta(y,t):=\int_1^t\tau^{-1}V_g(y,\tau)\,d\tau,\qquad h(y,t):=g(y,t)e^{-i\lambda\theta(y,t)}.
\end{equation}
The potential is locally bounded in time with values in $L^\infty$: indeed, splitting the Coulomb kernel into $|x|<1$ and $|x|\geq1$ gives $\|V_g\|_{L^\infty}\lesssim\|g\|_{L^{p_*}}^2+\|g\|_{L^2}^2$, because $p_*>3$. Thus the phase is well-defined and real. We have
\begin{equation}\label{eq:frac_h_equation}
i\partial_t h=t^{-2}e^{-i\lambda\theta}\Delta g,\qquad \|h(t)\|_{L^2}=\|u_0\|_{L^2}.
\end{equation}

\begin{lemma}\label{lem:frac_phase_bounds}
For $t\geq1$,
\begin{equation}\label{eq:frac_phase_bounds}
\|\nabla\theta(t)\|_{L^3}+\|\Delta\theta(t)\|_{L^{\frac{3}{2}}}+\||\nabla|^\beta\theta(t)\|_{L^{{\frac{3}{\beta}}}}\lesssim\varepsilon_0^2\log t.
\end{equation}
\end{lemma}
\begin{proof}
The Hardy-Littlewood-Sobolev inequality and $-\Delta V_g=|g|^2$ give
\[
\|\nabla V_g(t)\|_{L^3}+\|\Delta V_g(t)\|_{L^{\frac{3}{2}}}+\||\nabla|^\beta V_g(t)\|_{L^{{\frac{3}{\beta}}}}\lesssim\|g(t)\|_{L^3}^2\lesssim\varepsilon_0^2.
\]
Integrating these bounds with measure $\frac{d\tau}{\tau}$ proves the result.
\end{proof}

We now proceed with a similar idea as in Section \ref{sec:asymptotics}, except that, here we start with the spatial asymptotics in $H^{-2}$ instead of $H^{-1}$. 
\begin{proposition}\label{prop:frac_Hminus2}
There exists $h_\infty\in H^{-2}$ such that
\begin{equation}\label{eq:frac_Hminus2}
\|h(t)-h_\infty\|_{H^{-2}}\lesssim\varepsilon_0t^{-1}(1+\log t)^2,\qquad t\geq1.
\end{equation}
\end{proposition}
\begin{proof}
Let $\varphi\in H^2(\R^3)$. We compute, using \eqref{eq:frac_h_equation}
\[
\partial_t\la h(t),\varphi\ra=-it^{-2}\la e^{-i\lambda\theta}\Delta g,\varphi\ra=-it^{-2}\la g,\Delta(e^{i\lambda\theta}\varphi)\ra.
\]
Leibniz's rule and Sobolev embedding give
\begin{align*}
\big|\partial_t\la h(t),\varphi\ra\big|&\lesssim_\lambda t^{-2}\Big[\|g\|_{L^2}\big(\|\Delta\varphi\|_{L^2}+\|\nabla\theta\|_{L^3}\|\nabla\varphi\|_{L^6}\big)+\|g\|_{L^3}\|\varphi\|_{L^\infty}\big(\|\Delta\theta\|_{L^{3/2}}+\|\nabla\theta\|_{L^3}^2\big)\Big]\\
&\lesssim\varepsilon_0t^{-2}(1+\log t)^2\|\varphi\|_{H^2}.
\end{align*}
Thus $\|\partial_t h(t)\|_{H^{-2}}\lesssim\varepsilon_0t^{-2}(1+\log t)^2$. Integrating this bound from $t$ to infinity proves the existence of the limit and \eqref{eq:frac_Hminus2}.
\end{proof}

\begin{lemma}\label{lem:frac_h_growth}
For $t\geq1$,
\begin{equation}\label{eq:frac_h_growth}
\|h(t)\|_{H^\beta}\lesssim\varepsilon_0t^\delta(1+\log t).
\end{equation}
\end{lemma}
\begin{proof}
The fractional chain rule for $0<\beta<1$ and the reality of $\theta$ give
\[
\big\||\nabla|^\beta(e^{-i\lambda\theta})\big\|_{L^{{\frac{3}{\beta}}}}\lesssim_\beta |\lambda|\||\nabla|^\beta\theta\|_{L^{{\frac{3}{\beta}}}}\lesssim\varepsilon_0^2\log t.
\]
The derivative is homogeneous, so it is equivalently applied to $e^{-i\lambda\theta}-1$. The fractional Leibniz rule and \eqref{eq:frac_Sobolev_HLS} now imply
\begin{align*}
\||\nabla|^\beta h(t)\|_{L^2}&\lesssim\||\nabla|^\beta g(t)\|_{L^2}+\big\||\nabla|^\beta(e^{-i\lambda\theta(t)})\big\|_{L^{{\frac{3}{\beta}}}}\|g(t)\|_{L^{p_*}}\lesssim\varepsilon_0t^\delta(1+\log t).
\end{align*}
Together with \eqref{eq:frac_h_equation}, this proves the lemma.
\end{proof}

\begin{proposition}\label{prop:frac_L2_convergence}
We have $h_\infty$ belongs to $H^{\frac{1}{2}}$ and satisfies $\|h_\infty\|_{H^{\frac{1}{2}}}\lesssim\varepsilon_0$. Moreover, with $\rho=\frac{1-4\delta}{5}>0$,
\begin{equation}\label{eq:frac_L2_convergence}
\|h(t)-h_\infty\|_{L^2} \lesssim\varepsilon_0t^{-\rho}(1+\log t)^{\frac{6}{5}},\qquad t\geq1.
\end{equation}
In particular, $\|h_\infty\|_{L^2}=\|u_0\|_{L^2}$.
\end{proposition}
\begin{proof}
Set $d_n=h(2^{n+1})-h(2^n)$. Proposition \ref{prop:frac_Hminus2} and Lemma \ref{lem:frac_h_growth} give
\[
\|d_n\|_{H^{-2}}\lesssim\varepsilon_0 2^{-n}(1+n)^2,\qquad \|d_n\|_{H^\beta}\lesssim\varepsilon_0 2^{n\delta}(1+n).
\]
Interpolating with $\vartheta=\frac{5}{2\beta+4}$, for which $\frac12=-2(1-\vartheta)+\beta\vartheta$, yields
\[
\|d_n\|_{H^{\frac{1}{2}}}\lesssim\|d_n\|_{H^{-2}}^{1-\vartheta}\|d_n\|_{H^\beta}^{\vartheta}\lesssim\varepsilon_0 2^{-n\frac{\beta-\frac12-\frac52\delta}{\beta+2}}(1+n)^{2-\vartheta}.
\]
The exponent of decay is positive by \eqref{eq:frac_parameters}. Therefore $h(1)+\sum_{n\geq0}d_n$ converges in $H^{\frac{1}{2}}$ to the same limit $h_\infty$, with norm bounded by $C\varepsilon_0$.

Since $\beta>\frac12$, Lemma~\ref{lem:frac_h_growth} also gives $\|h(t)-h_\infty\|_{H^{\frac{1}{2}}}\lesssim\varepsilon_0t^\delta(1+\log t)$. Interpolating this bound with \eqref{eq:frac_Hminus2}, we obtain
\begin{align*}
\|h(t)-h_\infty\|_{L^2}&\lesssim\|h(t)-h_\infty\|_{H^{-2}}^{\frac{1}{5}}\|h(t)-h_\infty\|_{H^{\frac{1}{2}}}^{\frac{4}{5}}\\
&\lesssim\varepsilon_0t^{-\frac{1-4\delta}{5}}(1+\log t)^{\frac{6}{5}}.
\end{align*}
This proves \eqref{eq:frac_L2_convergence}; the mass identity follows from strong $L^2$ convergence and \eqref{eq:frac_h_equation}.
\end{proof}

\medskip

Define $\phi_\infty=(-\Delta)^{-1}|h_\infty|^2$. The endpoint Coulomb estimate in \eqref{eq:endpoint-Coulomb-Lorentz} gives $\|\phi_\infty\|_{L^{3,\infty}}\lesssim\varepsilon_0^2$. Since $|g|=|h|$, Proposition \ref{prop:frac_L2_convergence} also gives
\begin{align}\label{eq:frac_potential_convergence}
\|V_g(t)-\phi_\infty\|_{L^{3,\infty}}&\lesssim\big\||h(t)|^2-|h_\infty|^2\big\|_{L^1}\lesssim\|h(t)-h_\infty\|_{L^2}\big(\|h(t)\|_{L^2}+\|h_\infty\|_{L^2}\big)\notag\\
&\lesssim\varepsilon_0^2t^{-\rho}(1+\log t)^{\frac{6}{5}}.
\end{align}
Consequently, the finite phase correction
\begin{equation}\label{eq:frac_finite_phase}
\ell(y):=\int_1^\infty\tau^{-1}\big(V_g(y,\tau)-\phi_\infty(y)\big)\,d\tau
\end{equation}
is well-defined almost everywhere and belongs to $L^{3,\infty}$. The corresponding tail estimate then follows:
\begin{equation}\label{eq:frac_phase_tail}
\|\theta(t)-\phi_\infty\log t-\ell\|_{L^{3,\infty}}\lesssim\varepsilon_0^2t^{-\rho}(1+\log t)^{\frac{6}{5}}.
\end{equation}

\begin{proof}[Proof of Theorem \ref{thm:full_fractional}]
Global well-posedness follows from the full-space local theory and conservation of mass. As in the proof of Theorem \ref{thm:dom}, the inequality $|e^{ia}-e^{ib}|\lesssim|a-b|^{\frac{1}{2}}$, Sobolev embedding and H\"older inequality give
\begin{align*}
\|g(t)-h_\infty e^{i\lambda(\phi_\infty\log t+\ell)}\|_{L^2}
&\lesssim_\lambda\|h(t)-h_\infty\|_{L^2}+\|h_\infty\|_{H^{\frac{1}{2}}}\|\theta(t)-\phi_\infty\log t-\ell\|_{L^{3,\infty}}^{\frac{1}{2}}\\
&\lesssim\varepsilon_0t^{-\eta},
\end{align*}
where we used \eqref{eq:frac_L2_convergence}, \eqref{eq:frac_phase_tail} and $\rho>2\eta$. Absorbing $e^{i\lambda\ell}$ into $h_\infty$ and returning to the original variables in \eqref{eq:frac_profile} proves \eqref{eq:fractional_asymptotics}.
\end{proof}

\begin{remark}
We expect that the weaker localization assumption considered here can also be treated for the Robin half-space problem. Such an extension would require additional estimates adapted to the boundary and is not pursued in this paper since the key contribution is to develop a physical-space argument for \eqref{Hartree}.
\end{remark}


\end{document}